\documentclass[11pt,a4paper,reqno]{amsart}
\usepackage{amsmath}
\date{}
 \usepackage{amssymb,latexsym, esint}
 \usepackage{amsthm}
 \usepackage[latin1]{inputenc}
 
 \usepackage{ulem}

\usepackage[dvips]{graphicx}
\usepackage{color}

\newtheorem{defi}{Definition}[section]
\newtheorem{thm}{Theorem}[section]
\newtheorem{ex}{Example}[section]
\newtheorem{rem}{Remark}[section]
\newtheorem{rems}{Remarks}[section]

\newtheorem{cor}{Corollary}[section]

\newcommand{\R}{\mathbb R}
\newcommand{\N}{\mathbb N}
\newcommand{\Z}{\mathbb Z}

\begin{document}

\title[Sums of Eigenvalues]
{On sums of eigenvalues of elliptic operators on manifolds}

\author{Ahmad El Soufi}
\address{Universit\'e de Tours, Laboratoire de Math\'ematiques
et Physique Th\'eorique, UMR-CNRS 6083, Parc de Grandmont, 37200
Tours, France.} \email{elsoufi@univ-tours.fr}

\author{Evans M. Harrell II}
\address{School of Mathematics,
Georgia Institute of Technology,\
Atlanta GA 30332-0160, USA.} \email{ harrell@math.gatech.edu}

\author{Sa\"{\i}d Ilias}
\address{Universit\'e de Tours, Laboratoire de Math\'ematiques
et Physique Th\'eorique, UMR-CNRS 6083, Parc de Grandmont, 37200
Tours, France.} \email{ilias@univ-tours.fr}

\author{Joachim Stubbe}
\address{
EPFL, MATH-GEOM,
Station 8,
CH-1015 Lausanne, Switzerland}  \email{Joachim.Stubbe@epfl.ch}

\thanks{}

\begin{abstract}

We use the averaged variational principle introduced in a
recent article on graph spectra \cite{HaSt13}
to obtain upper bounds 
{for sums}
of eigenvalues of several partial differential operators of interest in geometric
analysis, which are analogues of 
{Kr\"oger's bound for Neumann spectra of Laplacians
on Euclidean domains \cite{Kro}.}
Among the operators
we consider are the Laplace-Beltrami operator on compact subdomains of manifolds.
{These estimates become more explicit and asymptotically sharp 
when the manifold is}
conformal to homogeneous spaces
(here extending a result
of Strichartz \cite{Str} with a simplified proof).
{In addition we obtain results for
the Witten Laplacian on the same sorts of domains and for}
Schr\"odinger operators with confining potentials on infinite Euclidean domains.
Our bounds have the sharp asymptotic form expected from the Weyl law or classical phase-space analysis.  
{Similarly} sharp bounds for the trace of the heat kernel follow as corollaries.
\end{abstract}

\subjclass[2010]{58J50, 35P15, 47A75}
\keywords{Manifold with density; Weighted Laplacian; Schr\"odinger operator, Witten Laplacian; Eigenvalue;
Upper bound, phase space, Weyl law, homogeneous space, conformal}

\maketitle

\section{Introduction}

In this article we consider the eigenvalues of 
self-adjoint, second-order elliptic
partial differential  operators defined on a subdomain of
a Riemannian manifold $(M,g)$ of dimension $\nu\ge 2$.  
{The model for the operators we
are able to treat is}
the Laplacian on a domain with Neumann boundary conditions,
defined in the weak sense, i.e. via 
the Laplacian energy 

\begin{equation}\label{Lapformsetup}\nonumber
\frac{\int_\Omega{(|\nabla^g \varphi({\bf x})|^2 } dv_g}{\int_\Omega{|\varphi({\bf x})|^2 dv_g}}
\end{equation}
on functions $\varphi \in H^1(\Omega)$,
but the class treated includes a large variety of
Schr\"odinger operators, even with weights.
Specifically,
the eigenvalues we shall discuss are operationally
defined by the min-max
procedure applied to expressions of the general form
\begin{equation}\nonumber
\mathcal{R}(\varphi) := \frac{\int_\Omega{(|\nabla^g \varphi({\bf x})|^2 + V({\bf x})|\varphi({\bf x})|^2 )}
 e^{-2 \theta({\bf x})} dv_g}{\int_\Omega{|\varphi({\bf x})|^2 e^{-2 \rho({\bf x})} dv_g}}.
\end{equation}
For convenience we set $w= e^{2(\rho-\theta)}$ so that $ \mathcal{R} $ 
takes on the form
\begin{equation}\label{basicsetup}
\mathcal{R}(\varphi) = \frac{\int_\Omega{(|\nabla^g \varphi({\bf x})|^2 + V({\bf x})|\varphi({\bf x})|^2 )}
w({\bf x}) e^{-2 \rho({\bf x})} dv_g}{\int_\Omega{|\varphi({\bf x})|^2 e^{-2 \rho({\bf x})} dv_g}}.
\end{equation}
Here $\rho \in C^1(\Omega)$, 
{$0 < C \le w({\bf x}) \in C^0(\Omega)$,}
and $V \in {\rm Lip}(\Omega)$
are real-valued functions. 
We define the Neumann eigenvalues of
\eqref{basicsetup} by the min-max principle \cite{BeBe,Th3}, {\it i.e.},
\begin{equation}\label{NeumannDef}
\mu_\ell := \max_{\{{\rm subspace \,} \mathfrak{S}: \, \dim(\mathfrak{S}) = \ell\}} \,\,\,
\min_{\{\varphi \in H^1(\Omega): \varphi \perp \mathfrak{S}, \|\varphi\|_{L^2} = 1\}} {\mathcal{R}(\varphi)}.
\end{equation}

Of course,  $\mu_\ell $ depends on the domain $\Omega$ as well as the choice of the metric $g$, 
{the density $ e^{-2 \rho}$ and weight $w$, and the potential $V$, but dependence on these 
quantities will not be indicated explicitly unless necessary.}
 
Under suitable regularity assumptions on $\Omega$ and $V$, the sequence $\{\mu_\ell\}$ is nothing but the spectrum of   the eigenvalue problem 

\begin{equation}\label{Operatror}
H\varphi =\mu\varphi  \qquad \mbox{in }\ \Omega
\end{equation}
with Neumann boundary conditions if $\partial \Omega \ne \emptyset$, where
\begin{eqnarray}\label{Operatror1}
 H\varphi &=& - e^{2 \rho}\mbox {div}_g\left( w e^{-2 \rho} \nabla^g \varphi\right)   + V w\varphi\\
 &=&  w \left\{ \Delta_g\varphi   + 2 \left( \nabla^g \theta, \nabla^g \varphi \right)_g + V \varphi  \right\}
 \end{eqnarray}
where $\Delta_g\varphi :=-\mbox {div}_g( \nabla^g \varphi)$ is the Laplace Beltrami operator associated with $g$. 

\smallskip

In the following sections we
derive semiclassically sharp phase-space upper bounds for the sums of the first $k$ eigenvalues
associated with \eqref{basicsetup}.
We also obtain bounds for the corresponding Riesz means and  heat trace.
The following inequalities, which are valid for any bounded  domain $\Omega\subset\R^\nu$, provide a sampling of these bounds :
\begin{equation}\nonumber
\frac{1}{k} \sum_{j=0}^{k-1} \mu_j \le
\frac{4 \pi^2 \nu}{\nu+2}  \left(\frac{k}{|\Omega| \omega_\nu}\right)^{\frac 2 \nu}    \fint_\Omega{w({\bf x}) d^{\nu}x}  +  \fint_\Omega \widetilde{V}({\bf x}){w({\bf x}) d^{\nu}x}   
\end{equation}
and 
\begin{equation}\nonumber
 \sum_{j\ge0} e^{- t\mu_j } \ge    \frac{  \vert\Omega\vert}{(4\pi t)^{\frac\nu 2} } \left(
  \fint_\Omega  w({\bf x})\,d^{\nu}x\right)^{-\frac\nu 2} e^{- t  \fint_\Omega  \widetilde{V}({\bf x}) w({\bf x})\,d^{\nu}x}  ,
\end{equation}
where $\widetilde{V}({\bf x}) := V({\bf x}) + |\nabla \rho|^2({\bf x})$, $\vert\Omega\vert$ is the volume of $\Omega$,  $\omega_\nu$ is the volume of the unit ball in  $\R^\nu$ and, for every  $f\in L^1(\Omega)$, $ \fint_\Omega f({\bf x}) d^{\nu}x=\frac{1}{|\Omega|}\int_\Omega{f({\bf x}) d^{\nu}x}$ is the mean value of $f$ with respect to Lebesgue measure.

\smallskip
When appropriate we remark on the simpler consequences that 
apply under assumptions on 
{$ \rho,\ w$, and $V$.
The case where $V=\rho=0$ and $w=1$ identically, and $M = \R^\nu$
reduces to the situation treated by Kr\"oger
in his ground-breaking work \cite{Kro}}, and this result
was already extended to subdomains of general
homogeneous spaces  by
Strichartz \cite{Str} (see also  \cite{Gal}).  The upper bounds in \cite{Kro,  Str}
are notable for being sharp in the sense of agreeing with the
``semiclassical'' Weyl law,
with the optimal constant.  For the background and context of Weyl-sharp bounds
on sums of Laplacian eigenvalues, we refer to \cite{Lap1,Lap2}.

\smallskip
In this article 
a new, simplified proof is used, and
we considerably enlarge the family of self-adjoint elliptic operators for which semiclassical upper
bounds are proved.
Even when
$V=0$, new cases of interest that are treated include
the Witten Laplacian,
for which $w = 1$; the Laplacian of a conformal metric $\tilde g = \alpha^{-2} g$, for which $e^{-2 \rho} =  \alpha^{-n}$ and 
$w =   \alpha^{2}$;
and the vibrating membrane with
variable density $\gamma({\bf x})$, for which $\gamma({\bf x}) =e^{-2 \rho}$ and $ w =e^{2 \rho}$.

\smallskip
In the last part of the paper, we focus on domains of compact homogeneous Riemannian spaces. We revisit  the inequality of Strichartz (\cite[Theorem 2.2]{Str})  in the light of this new approach and obtain extensions of Strichartz's inequality to the case where the Laplace operator is penalized by a potential in 
the presence of weights. For example, we prove that if $\Omega$ is a   domain of a compact homogeneous Riemannian manifold $(M,g)$, then the eigenvalues $\mu_l$ associated with  \eqref{basicsetup} 
on $\Omega $ satisfy
\begin{equation}\nonumber
 \sum_{j\ge 0} \left(z- \mu_j\right)_+ \ge 
 \frac{ \vert\Omega\vert_g}{\vert M\vert_g}  \sum_{j\ge 0}  \left( z- \tilde\lambda_j   \right)_+
\end{equation}
for all $z\in\R$, and
\begin{equation}\nonumber
 \sum_{j\ge 0} e^{-\mu_j t}\ge \frac{\vert \Omega\vert_g}{\vert M\vert_g}  \sum_{j\ge 0} e^{-\tilde\lambda_j t}
\end{equation}
for all $t>0$, where $ \tilde\lambda_j  = \lambda_j \fint_\Omega w \, dv_g +\fint_\Omega \widetilde{V}w \, dv_g  $, and where the $\lambda_j 's$ are the eigenvalues of the Laplacian on the whole manifold $M$ (see Theorem \ref{gen-homog} and Corollary \ref{heat-homog}).  The extension (stricto sensu) of Strichartz inequality is given in Theorem \ref{sum-homog}
and takes the following form when $\Omega=M$:
\begin{equation}\nonumber
 \sum_{j=0}^{k-1} \mu_j \le \sum_{j=0}^{k-1} \tilde\lambda_j.
 \end{equation}

It is known that without assumptions of regularity, these
variationally defined Neumann eigenvalues for the Laplacian
may have finite points of
accumulation of a quite arbitrary sort, as entertainingly discussed in \cite{HeSeSi}.
In this case the definition \eqref{NeumannDef} would yield
$\mu_\ell = \inf(\sigma_{\rm ess})$ for all $\ell$ greater than some value, and the bounds we shall provide would
become uninteresting.  We note that, for example, the spectrum of the Neumann Laplacian is guaranteed
to have no finite points of accumulation if the boundary is piecewise smooth \cite{HeSeSi}.

\begin{rem}
Before closing this section, we make some further technical remarks about how
to define the Dirichlet and Neumann problems for these elliptic operators
in the weak, or quadratic-form, sense.  In this regard we follow
Edmunds and Evans \cite{EdEv}, where in Chapter VII it is shown that uniformly elliptic quadratic forms, on arbitrary open sets in Euclidean space, determine unique operators via the Friedrichs extension, which, when the domain is sufficiently regular, reduce to the classically defined
operators for the Dirichlet and Neumann problems.  (See also \cite{Wei,ReSi}.)
In particular, defining the quadratic form
\eqref{basicsetup} initially on the Sobolev space $W_0^{1,2}(\Omega)$ corresponds
to Dirichlet boundary conditions, whereas defining it initially on the restrictions to $\Omega$ of functions in the space
$W_0^{1,2}(\R^\nu)$ corresponds to Neumann conditions.  (For domains allowing
a Sobolev extension property the latter set coincides with $W^{1,2}(\Omega)$.)  It is not in general possible to say that the operators thus defined satisfy
boundary conditions in a classical sense, or to guarantee regularity at the boundary.  However, in cases where the boundary is sufficiently regular, integration by parts transforms expressions $\left\langle H \varphi, \varphi\right\rangle$ where $H$ is a classically defined operator into a
quadratic form of the type \eqref{basicsetup} for $\varphi$ in a dense
subset of the Sobolev spaces
corresponding to Dirichlet or respectively Neumann conditions.

There are certainly significant questions of
regularity of the eigenfunctions in the case when $\Omega$ is an arbitrary open set, treated for example in \cite{Agm}, but
they will play no role in the present article.

The extension of the analysis of \cite{EdEv} from $\R^\nu$ to  manifolds is straightforward, because only Hilbert-space structures and locally defined properties of functions and their gradients are used.
\end{rem}

\section{The averaged variational principle}\label{Prel}

In this section we recall the averaged variational principle which will be foundational for this article.  The following is mainly a restatement of Theorem 3.1 of  Harrell-Stubbe \cite{HaSt13}, along
with a characterization of the case of equality.

\begin{thm}\label{AvePrin}
Consider a self-adjoint operator $H$ on a Hilbert space $\mathcal{H}$,
the spectrum of which is discrete at least in its lower portion, so that
$- \infty < \mu_0 \le \mu_1 \le \dots$.
The corresponding orthonormalized eigenvectors
are denoted $\{\mathbf{\psi}^{(\ell)}\}$.  The
closed quadratic form corresponding to $H$
is denoted $Q(\varphi, \varphi)$ for vectors $\varphi$ in
the quadratic-form domain $\mathcal{Q}(H) \subset \mathcal{H}$.
Let $f_\zeta \in \mathcal{Q}(H)$ be
a family of
vectors indexed by
a variable $\zeta$ ranging over
a measure space $(\mathfrak{M},\Sigma,\sigma)$.
Suppose that $\mathfrak{M}_0$ is a
subset of $\mathfrak{M}$.  Then for any  $z\in \R$, 
\begin{equation}\label{RieszVersion}
	\sum_{j}{\left(z - \mu_j\right)_{+} \int_{\mathfrak{M}}\left|\langle\mathbf{\psi}^{(j)}, f_\zeta\rangle\right|^2\,d \sigma}
       \geq
	\int_{\mathfrak{M}_0}{\left(z\| f_\zeta\|^2 - Q(f_\zeta,f_\zeta) \right) d \sigma},
\end{equation}
provided that the integrals converge. Moreover, equality holds in  \eqref{RieszVersion} for $z\in\R$ if and only if
up to sets of measure $0$,
$$\left\{f_\zeta \ ; \ \zeta \in {\mathfrak{M}_0} \right\}\subset E(z)\quad\mbox{and} \quad \left\{f_\zeta \ ; \ \zeta \in \mathfrak{M}\setminus\mathfrak{M}_0 \right\}\perp E_0(z),$$
where $E(z)=\bigoplus_{\mu\le z}\ker(H-\mu I)$ and $E_0(z)=\bigoplus_{\mu< z}\ker(H-\mu I)$.
\end{thm}
Taking $z=\mu_k$ in \eqref{RieszVersion} we obtain 
\begin{equation}\label{kroeger-2}
\begin{split}
    \mu_{k} \bigg(\int_{\mathfrak{M}_0}\|f_\zeta\|^2 \,d \sigma&-\sum_{j=0}^{k-1}\int_{\mathfrak{M}}|\langle\mathbf{\psi}^{(j)}, f_\zeta\rangle|^2\,d \sigma\bigg)\\
&\leq \int_{\mathfrak{M}_0}{Q(f_\zeta,f_\zeta) d \sigma}-\sum_{j=0}^{k-1}\mu_j\int_{\mathfrak{M}}|\langle\mathbf{\psi}^{(j)}, f_\zeta\rangle|^2\,d \sigma,\\
\end{split}
\end{equation}

\begin{rems}

{1.  The averaged variational principle is an abstract version and sharpening 
of ideas appearing in various place in the literature, including not only \cite{Kro}, but also 
work of Lieb and others on coherent states and trace inequalities \cite{Lie,LiSo}.
In special cases, similar use
of averaging and tight frames for the study of eigenvalue sums and related quantities 
has also been  made by Laugesen and Siudeja \cite{LaSi}.}

\smallskip
\noindent 2.  We point out that the normalization of the test function $f_\zeta$ 
could be incorporated into the measure, so that, for example, Eq. \eqref{RieszVersion} could alternatively be written in terms of integrals of expectation values such as
\begin{equation}\label{RieszVersionAlt}
\int\left(\frac{|\langle\mathbf{\psi}^{(j)}, f_\zeta\rangle|^2}{\|f_\zeta\|^2}\right)\,d \sigma,
\end{equation}
{\rm i.e.}, over norms of projections of the eigenfunctions.  Despite the suggestiveness of these
alternatives, an advantageous feature of \eqref{kroeger-2}-\eqref{RieszVersion} that
we shall later exploit is that useful identities are available for averages of norms of 
some choices of $f_\zeta$.
Still, if the test functions $f_\zeta$ and the measure space $\mathfrak{M}$ 
constitute a {\rm tight frame}, in the sense of satisfying a generalized Parseval identity 
\cite{HeWa}, then alternative forms of the inequalities imply appealing
variational principles for sums and Riesz means of eigenvalues, as captured in the next corollary.

\end{rems}

\begin{cor}\label{Aveprin1}
Under the assumptions of the Theorem, 
{suppose further that
$f_\zeta$ is a nonvanishing family of test functions with the property that 
for all $\phi \in \mathcal{H}$, 
$$\int_{\mathfrak{M}}\frac{|\langle\mathbf{\phi}, f_\zeta\rangle|^2}{\|f_\zeta\|^2}\,d \sigma = A \|\mathbf{\phi}\|^2$$
for a fixed constant $A > 0$.  Then for any $\mathfrak{M}_0 \subset \mathfrak{M}$ such that $\left(|\mathfrak{M}_0|- A \, k \right)\mu_k\ge 0$,
\begin{equation}\label{withParseval}
\frac{1}{k} \sum_{j=0}^{k-1}{\mu_j} \le \frac 1 {|\mathfrak{M}_0|} \int_{\mathfrak{M}_0}{\left(\frac{Q(f_\zeta,f_\zeta)}{\|f_\zeta\|^2}\right)} d \sigma.
\end{equation}
For Riesz means,
\begin{equation}
	\sum_{j}{\left(z - \mu_j \right)_{+}}   \geq
	\frac{1}{A} \int_{\mathfrak{M}_0}{\left(z - \frac{Q(f_\zeta,f_\zeta)}{\| f_\zeta\|^2} \right) d \sigma}.
\end{equation}
}
\end{cor}
\noindent
The proof of Corollary \ref{Aveprin1}
is immediate; see \cite{HaSt13} for more in this connection.  To make our exposition self-contained, we provide here the proof of
the inequality \eqref{RieszVersion} in Theorem \ref{AvePrin} before discussing the case of equality.

\begin{proof}[Proof of Theorem \ref{AvePrin}]
For every integer $l\ge 0$, we denote by $P_{l} $ the orthogonal projector onto the subspace spanned by $\{{\psi}^{(j)}\ , \ j\le l\}$, i.e. $P_{l} f= \sum_{j=0}^l \langle {\psi}^{(j)}, f\rangle {\psi}^{(j)}$. Let  $z\in\R$, $z> \mu_0$ (the inequality \eqref{RieszVersion} being obvious for $z\le \mu_0$), and let $k$ be the smallest integer such that $z\le \mu_k$ (that is $z\in(\mu_{k-1}, \mu_k]$). Then
\begin{equation}\label{rayleigh-quot}
z\Vert f-P_{k-1} f\Vert^2\le\mu_k \Vert f-P_{k-1} f\Vert^2\le Q(f-P_{k-1} f,  f-P_{k-1} f ),
\end{equation}
and, after direct computations,
$$z\left(\Vert f\Vert^2-\Vert P_{k-1} f\Vert^2\right)\le Q(f,  f ) -Q(P_{k-1} f,  P_{k-1} f ).$$
{
With $\Vert P_{k-1} f\Vert^2=\sum_{j=0}^{k-1} \langle {\psi}^{(j)}, f\rangle^2$ and $Q(P_{k-1} f,  P_{k-1} f )=\sum_{j=0}^{k-1}\mu_j \langle {\psi}^{(j)}, f\rangle^2$, 
this yields}
$$z \Vert f\Vert^2-Q(f,  f ) \le \sum_{j=0}^{k-1} (z-\mu_j)\langle {\psi}^{(j)}, f\rangle^2 .$$
Applying this last inequality
to $f_\zeta$, $\zeta\in\mathfrak{M}_0$, and integrating over $\mathfrak{M}_0$ we get
\begin{equation}\label{kroeger-3}
\begin{split}
z \int_{\mathfrak{M}_0}\Vert f_\zeta\Vert^2 d \sigma-\int_{\mathfrak{M}_0}Q( f_\zeta,  f_\zeta ) d \sigma &\le \sum_{j=0}^{k-1} (z-\mu_j)\int_{\mathfrak{M}_0}|\langle {\psi}^{(j)}, f_\zeta\rangle|^2  d \sigma\\
&= \sum_{j\ge 0} (z-\mu_j)_+\int_{\mathfrak{M}_0}|\langle {\psi}^{(j)}, f_\zeta\rangle|^2  d \sigma.
\end{split}
\end{equation}
The inequality \eqref{RieszVersion} follows from \eqref{kroeger-3} and the obvious inequality
\begin{equation}\label{obv-ineq}
\sum_{j\ge 0} (z-\mu_j)_+\int_{\mathfrak{M}_0}|\langle {\psi}^{(j)}, f_\zeta\rangle|^2  d \sigma\le  \sum_{j\ge 0} (z-\mu_j)_+\int_{\mathfrak{M}}|\langle {\psi}^{(j)}, f_\zeta\rangle|^2  d \sigma.
\end{equation}

Assume now that equality holds in \eqref{RieszVersion}. This implies  that equality holds in \eqref{obv-ineq} and in \eqref{rayleigh-quot} for $f=f_\zeta$ for almost all  $\zeta\in\mathfrak{M}_0$. Equality in \eqref{rayleigh-quot} holds for  $f$ either when $z<\mu_k$ and $f=P_{k-1} f$ or if $z=\mu_k$ and $H(f-P_{k-1} f)=\mu_k(f-P_{k-1} f)$, which implies in both cases that $f\in E(z)$.  On the other hand, equality in \eqref{obv-ineq}  implies that, for almost all $\zeta\in\mathfrak{M}\setminus\mathfrak{M}_0$ and all $j\in\N$ such that $\mu_j<z$, $f_\zeta$ is orthogonal to $\mbox{span}\{{\psi}^{(0)},\dots,{\psi}^{(j)} \}$, which means that $f_\zeta$ is orthogonal to $E_0(z)$. 

Conversely, under the conditions of the statement, 
$$\sum_{j}{\left(z - \mu_j\right)_{+} \int_{\mathfrak{M}}\left|\langle\mathbf{\psi}^{(j)}, f_\zeta\rangle\right|^2\,d \sigma}
= \sum_{j}{\left(z - \mu_j\right)_{+} \int_{\mathfrak{M}_0}\left|\langle\mathbf{\psi}^{(j)}, f_\zeta\rangle\right|^2\,d \sigma}$$
$$=\sum_{\mu_j\le z}  z\int_{\mathfrak{M}_0}\left|\langle\mathbf{\psi}^{(j)}, f_\zeta\rangle\right|^2\,d \sigma - \sum_{\mu_j\le z} \mu_j \int_{\mathfrak{M}_0}\left|\langle\mathbf{\psi}^{(j)}, f_\zeta\rangle\right|^2\,d \sigma $$
$$=  \sum_{j=0}^{+\infty} z\int_{\mathfrak{M}_0}\left|\langle\mathbf{\psi}^{(j)}, f_\zeta\rangle\right|^2\,d \sigma - \sum_{j=0}^{+\infty} \mu_j \int_{\mathfrak{M}_0}\left|\langle\mathbf{\psi}^{(j)}, f_\zeta\rangle\right|^2\,d \sigma $$
$$=\int_{\mathfrak{M}_0}\left( z \Vert f_\zeta\Vert^2 -Q( f_\zeta,f_\zeta)\right) \,d \sigma. $$
\end{proof}

\smallskip

As the guiding example for this article,
when $\Omega $  is a bounded subdomain of $\R^\nu$,  we may use test functions of the form
 $$
f_\zeta({\bf x}) := \frac{1}{(2 \pi)^{\nu/2}} e^{i {\bf p} \cdot {\bf x}},
$$
where $\zeta$ has been equated to ${\bf p}$, which ranges over
$\mathfrak{M} = \R^{\nu}$ with Lebesgue measure (The reason for distinguishing $\zeta$ logically from ${\bf p}$
will be made clear in Theorem \ref{PSBound}.). Indeed,
$\|f_\zeta\|^2 = \frac{|\Omega|}{(2 \pi)^\nu}$ for all $\zeta$, where $|\Omega|$ is the  Euclidean volume  of $\Omega$, 
 and Parseval's identity gives 
$$\int_{\R^{\nu}}{\left|\left\langle\phi, f_\zeta\right\rangle\right|^2 d^{\nu}p}
=\|\phi\|^2.
$$ 
Hence,   applying  Corollary \ref{Aveprin1} with $\mathfrak{M}_0\subset \mathfrak{M}$ taken to be  the Euclidean ball of radius $ 2\pi \left(\frac k{\vert \Omega\vert \omega_\nu}\right)^{\frac1 \nu} $,  we recover Kröger's inequality for Neumann eigenvalues of the Euclidean Laplacian (here $\omega_\nu$ stands for the  volume of the $\nu$-dimensional Euclidean unit ball). Indeed, in this case, the Rayleigh quotient of $f_\zeta$ is simply given by ${\mathcal R}(f_\zeta) = \vert {\bf p} \vert^2$ and \eqref{withParseval} yields
$$\frac{1}{k} \sum_{j=0}^{k-1} \mu_j \le  \frac 1 {|\mathfrak{M}_0|} \int_{\mathfrak{M}_0} \vert {\bf p} \vert^2 d^{\nu}p=
\frac{4 \pi^2 \nu}{\nu+2}  \left(\frac{k}{|\Omega| \omega_\nu}\right)^{\frac 2 \nu}. $$

This approach can be
{applied
to easily extend}
Kr\"oger's inequality to Neumann eigenvalues on a bounded subdomain of $\R^\nu$ in 
the presence of nontrivial potential and weights. 

\begin{ex}\label{noh}
Let $\mu_0 \le \mu_1 \le \dots$ be the variationally defined Neumann eigenvalues \eqref{NeumannDef}
on a bounded open set $\Omega \subset \R^\nu$ endowed with the standard Euclidean metric, where
$w$, $\rho$, and $V$ satisfy the assumptions stated above.
Then
\begin{equation}\label{EucUpperbdnoh}
\frac{1}{k} \sum_{j=0}^{k-1} \mu_j \le
\frac{4 \pi^2 \nu}{\nu+2}  \left(\frac{k}{|\Omega| \omega_\nu}\right)^{\frac 2 \nu}    \fint_\Omega{w({\bf x}) d^{\nu}x}  +  \fint_\Omega \widetilde{V}({\bf x}){w({\bf x}) d^{\nu}x}   ,
\end{equation}
where {$\widetilde{V}({\bf x}) := V({\bf x}) + |\nabla \rho|^2({\bf x})$ } and, for every  $f\in L^1(\Omega)$, $ \fint_\Omega f({\bf x}) d^{\nu}x=\frac{1}{|\Omega|}\int_\Omega{f({\bf x}) d^{\nu}x}$ is the mean value of $f$ with respect to Lebesgue measure.
\end{ex}
Example \ref{noh} sets the stage for a more general result that we obtain in Section \ref {Riem-sect} in the context of   Riemannian manifolds. We also stress that these estimates will be improved in Section \ref{phase-space-sect}, with the aid of a coherent-state
analysis relating the upper bounds to phase-space volumes.

\smallskip

Upper bounds for individual Neumann eigenvalues $\mu_j$
are also obtainable 
from the averaged variational principle. 
In order to somewhat simplify the bound, 
let us define the shifted Neumann eigenvalues
\begin{equation}\label{ev-normalized}
  \widetilde{\mu_j}:=\mu_j- \fint_\Omega\widetilde{V}({\bf x}){w({\bf x}) d^{\nu}x}  .
\end{equation}
In terms of these quantities, we will be able to show that (see Corollary \ref{mu_bd_general})
\begin{equation} \label{EucUpperbdnoh-cor}
  \widetilde{\mu_k}\leq 4 \pi^2 \left(\frac{k}{|\Omega| \omega_\nu}\right)^{\frac 2 \nu}
  \left(1+ 2\sqrt{\frac{1-S_k}{\nu+2}}\right)\fint_\Omega{w({\bf x}) d^{\nu}x},
\end{equation}
where
\begin{equation}
  S_k:=\frac{\frac{1}{k}\sum_{j=0}^{k-1} \widetilde{\mu_j}}{\frac{4 \pi^2\nu}{\nu+2} \left(\frac{k}{|\Omega| \omega_\nu}\right)^{\frac 2 \nu}\fint_\Omega{w({\bf x}) d^{\nu}x}}\leq 1.
\end{equation}

\section{Bounds for Neumann eigenvalues on domains of Riemannian manifolds}\label{Riem-sect}

Let $(M,g)$ be a Riemannian manifold of dimension $\nu\ge 2$ and let $\Omega $ be a bounded subdomain of $M$. Of course, when $M$ is a closed manifold, $\Omega$ can be equal to the whole of $M$. 

Let $F:(M,g)\to \R^N$, be an isometric embedding (whose existence for sufficiently large $N$ is guaranteed by Nash's embedding Theorem). 
To any function $u\in L^2(\Omega)$, we associate the function $\hat u_F :\R^N\to\R$ defined by
\begin{equation}\label{generalFT}
\hat u_F ({\bf p}) =\int_\Omega u({\bf x}) e^{i {\bf p}\cdot F({\bf x})} dv_g,
\end{equation}
where the dot stands for the Euclidean inner product in $\R^N$ (i.e., $\hat u_F $ is the Fourier transform of the signed measure $F_*(udv_g)$ supported by $F(\Omega)$). It is well-known, since  the works of Hörmander, Agmon-Hörmander and others 
(see \cite[Theorem 2.1]{AgHo} \cite[Theorem 7.1.26]{Hor}, \cite[Corollary 5.2]{Str1}), that  there exists a constant $C_{F(\Omega)}$ such that, $\forall u\in L^2(\Omega)$ and $\forall R>0$,
\begin{equation}\label{Horm}
 \int_{B_R}\vert \hat u_F ({\bf p})\vert^2 d^Np \le C_{F(\Omega)} R^{N-\nu} \Vert u\Vert^2,
 \end{equation}
where $B_R$ is the Euclidean ball of radius $R$ in $\R^N$ centered at the origin and $\Vert u\Vert^2 = \int_\Omega u^2 dv_g$.
{In other words the Fourier functions appearing in \eqref{generalFT}
constitute a frame that is not generally tight.}

We define the Riemannian constant $H_\Omega$ by 
\begin{equation}\label{def-Horm}
H_\Omega = \inf_{N\ge \nu}\ \inf_{F\in I(M,\R^N)} \left(\frac{\nu+2}{N+2}\right)^{\frac{\nu}{2}}\frac1{\omega_N}\ C_{F(\Omega)},
  \end{equation}
where  $I(M,\R^N)$ is the set of isometric embeddings from $(M,g)$ to $\R^N$
. 

When $\Omega$ is a domain of $\R^\nu$, we may take for $F$ the identity map so that, $\forall u\in L^2(\Omega)$, $\hat u_I$ is nothing but the Fourier transform of $u$ extended by zero outside $\Omega$.  Using Parseval's identity we get
$\forall R>0$,
$$ \int_{B_R}\vert \hat u_F ({\bf p})\vert^2 d^\nu p \le \int_{\R^\nu}\vert \hat u_F ({\bf p})\vert^2 d^\nu p = (2\pi)^\nu\Vert u\Vert^2.$$
Thus  $C_{I(\Omega)}= (2\pi)^\nu $ and 
\begin{equation}\label{H-Euc}
H_\Omega\le  \frac{(2\pi)^\nu}{\omega_\nu} .
\end{equation}
In all the sequel, the notation $\vert\Omega\vert_g$ will designate the Riemannian volume of $\Omega$ with respect to $g$.  We will also use the notation $\fint_\Omega f\,dv_g $ to represent the mean value of  a function $f\in L^1(\Omega)$ with respect to the Riemanian measure $dv_g$.
(I.e., $\fint_\Omega f\,dv_g =\frac{1}{\vert\Omega\vert_g}\int_\Omega f\,dv_g $.)

\begin{thm}\label{general}
Let $(M,g)$ be a Riemannian manifold of dimension $\nu\ge 2$.
Let $\mu_l=\mu_l(\Omega, g, \rho,w,V)$, $l\in\N$, be the eigenvalues defined by \eqref{NeumannDef}
on a bounded open set $\Omega \subset M$,  where
$w$, $\rho$, and $V$ satisfy the assumptions stated above.
Then

\smallskip

\noindent (1) For all $z\in\R$,
\begin{equation}\label{Lieb-Th_gen}
 \sum_{j\ge 0} \left(z- \mu_j\right)_+ \ge \frac{2\  \vert\Omega\vert_g}{(\nu+2) H_\Omega} \left(
  \fint_\Omega  w\,dv_g\right)^{-\frac\nu 2}\left( z- \fint_\Omega \widetilde{V}w \, dv_g    \right)_+^{1+\frac\nu 2},
\end{equation}
where $\widetilde{V} = V+ \vert\nabla^g\rho\vert^2$.

\smallskip

\noindent (2) For all $k\in\N$,
\begin{equation}\label{Kroger_gen}
\frac{1}{k} \sum_{j=0}^{k-1} \mu_j \le \frac{\nu}{\nu+2}   \left(\frac{H_\Omega }{|\Omega|_g}k\right)^{\frac 2 \nu} \fint_\Omega w\,dv_g
+\fint_\Omega  \widetilde{V} w\,dv_g .   
\end{equation}

\smallskip

\noindent (3) For all $t>0$,
\begin{equation}\label{heat_gen}
 \sum_{j\ge0} e^{- t(\mu_j - \fint_\Omega  \widetilde{V} w\,dv_g)} \ge  \left(\frac\pi{t}\right)^{\frac\nu2} \  \frac{  \vert\Omega\vert_g}{\omega_\nu H_\Omega} \left(
  \fint_\Omega  w\,dv_g\right)^{-\frac\nu 2}.   
\end{equation}
\end{thm}

\begin{proof}
Let $F:(M,g)\to\R^N$ be an isometric embedding. 
For simplicity, we identify the domain $\Omega$ with its image $F(\Omega)\subset R^N$ and any function $u:\Omega\to \R$ with $u\circ F^{-1}:F(\Omega)\to\R$.

\smallskip
We apply Theorem \ref{AvePrin}  using test functions of the form
$$
f_\zeta({\bf x}) := e^{i {\bf p} \cdot {\bf x} + \rho({\bf x})},
$$
where 
$\zeta$ has been equated to ${\bf p}$, which ranges over 
$\mathfrak{M}=B_R\subset \R^N$ endowed with Lebesgue measure, where $B_R$ is a Eucidean $N$-dimensional ball whose radius $R$ is to be determined later. Our Hilbert space here is $L^2(\Omega, e^{-2\rho } dv_g)$ (endowed with the norm $\Vert u\Vert^2= \int_\Omega u^2 e^{-2\rho } dv_g$). Hence, for all $\zeta$,
$$ \Vert f_\zeta\Vert^2 = |\Omega|_g$$
and consequently
\begin{equation}\label{term1}
\int_{\mathfrak{M}} \Vert f_\zeta\Vert^2 d^{N}p =  |\Omega|_g \omega_NR^N.
\end{equation}
On the other hand, in our case the quadratic form is 
$$Q(f_\zeta, f_\zeta) = \int_\Omega \left( \vert\nabla^\tau f_\zeta\vert^2+ V \vert f_\zeta\vert^2\right ) w({\bf x}) \ e^{-2\rho({\bf x}) } dv_g,$$
where $\nabla^\tau  f_\zeta$ is the tangential part of the gradient of $f_\zeta$ (more generally, for all $v\in \R^N$, $v^\tau$ will designate the tangential vector field induced on $\Omega$ by orthogonal projection of $v$). Thus, with $\vert f_\zeta\vert^2=e^{2\rho } $ and $\vert\nabla^\tau f_\zeta\vert^2 =\vert {\bf p}^\tau+\nabla^\tau\rho\vert^2$, 
$$Q(f_\zeta, f_\zeta) = \int_\Omega \left(  \vert{\bf p}^\tau \vert^2({\bf x}) +2{\bf p}\cdot\nabla^\tau\rho({\bf x})\right) w({\bf x})   dv_g+ \int_\Omega (V+\vert\nabla^\tau \rho\vert^2 )  w({\bf x}) dv_g.$$
Observe that for symmetry reasons, for all $v\in \R^N\setminus\{0\}$,
$$\int_{B_R}{\bf p}\cdot v\,d^{N}p =0$$
and, after elementary calculations,
$$\int_{B_R}({\bf p}\cdot v)^2\,d^{N}p = \frac{\vert v\vert}{N}\int_{B_R}\vert{\bf p} \vert^2\,d^{N}p = \frac{\vert v\vert}{N+2}\omega_NR^{N+2}.$$
Thus, if $\{v_1,\dots,v_\nu\}$ is an orthonormal basis of the tangent space of $\Omega$ at a point ${\bf x}$, then
$$ \int_\Omega  \vert{\bf p}^\tau \vert^2({\bf x})\,d^{N}p = \sum_{j\le\nu}\int_{B_R}({\bf p}\cdot v_j)^2\,d^{N}p = \frac{\nu}{N+2}\omega_NR^{N+2}.$$
This leads to 
\begin{equation}\label{term2}
\int_{\mathfrak{M}} Q(f_\zeta, f_\zeta)  d^{N}p =  \frac{\nu}{N+2}\omega_NR^{N+2}\int_\Omega w\,dv_g 
+ \omega_NR^{N}\int_\Omega  \widetilde{V} w\,dv_g .
\end{equation}
It remains to deal with the integrals $\int_{\mathfrak{M}}|\langle\mathbf{\psi}^{(j)}, f_\zeta\rangle|^2\,d^Np$, where  $\{\psi^{(j)}\}$ is an $L^2(\Omega, e^{-2\rho } dv_g )$-orthonormal basis of eigenfunctions associated to $\{\mu_j\} $. Setting $\phi^{(j)}=  e^{-\rho }{\psi}^{(j)}$,
$$\langle\mathbf{\psi}^{(j)}, f_\zeta\rangle = \int_\Omega f_\zeta {\psi}^{(j)}  e^{-2\rho } dv_g =  \int_\Omega e^{i {\bf p}\cdot{\bf x}} {\psi}^{(j)}  e^{-\rho } dv_g =\hat\phi_F^{(j)} ({\bf p}).$$
 Using \eqref{Horm} we obtain
  \begin{equation}\label{term3}
  \int_{B_R}|\langle\mathbf{\psi}^{(j)}, f_\zeta\rangle|^2\,d^Np = \int_{B_R}|\hat\phi_F^{(j)} ({\bf p})|^2\,d^Np\le C_{F(\Omega)} R^{N-\nu} \int_\Omega \vert\phi^{(j)}\vert^2 dv_g $$
 $$\qquad \qquad \qquad \qquad= C_{F(\Omega)} R^{N-\nu}\int_\Omega \vert\psi^{(j)}\vert^2 e^{-2\rho } dv_g = C_{F(\Omega)} R^{N-\nu}.
 \end{equation}


We put \eqref{term1}, \eqref{term2}, and \eqref{term3} into \eqref{RieszVersion} after choosing ${\mathfrak{M}}_0={\mathfrak{M}}=B_R$, and obtain for all $R>0$ and $z\in \R$
 \begin{equation}\label{Riesz-gen}
 \sum_{j\ge 0} \left(z- \mu_j\right)_+ C_{F(\Omega)}   \ge  |\Omega|_g \omega_NR^{\nu}\left(
 z - \frac{\nu R^{2}}{N+2} \fint_\Omega  w\,dv_g - \fint_\Omega  \widetilde{V} w\,dv_g\right).
\end{equation}
The right side of this inequality is optimized when $R=0$ if $z  \le \fint_\Omega  \widetilde{V} w\,dv_g $ and when $R^2=\frac{N+2}{\nu+2}\left(z  - \fint_\Omega  \widetilde{V} w\,dv_g\right)/\fint_\Omega   w\,dv_g$ otherwise. Thus 
  \begin{equation}\label{Riesz-gen1}
   \begin{split}
 \sum_{j\ge0} &\left(z  - \mu_j\right)_+ C_{F(\Omega)} \ge\\
   & |\Omega|_g \omega_N\left(\frac{N+2}{\nu+2}\right)^{\frac\nu 2}\frac{2}{\nu+2} \left(
  \fint_\Omega  w\,dv_g\right)^{-\frac\nu 2} \left(
 z  - \fint_\Omega  \widetilde{V} w\,dv_g\right)_+^{1+\frac\nu 2}.\\
 \end{split}
 \end{equation}
Taking the infimum with respect to $F$ and $N$ we get \eqref{Lieb-Th_gen}. 

\smallskip

To prove \eqref{Kroger_gen} we first observe that taking $z=\mu_k$ in \eqref{Riesz-gen} gives 
 \begin{equation}\label{Kro_gen1}
k\mu_k  -\sum_{j=0}^{k-1}  \mu_j   \ge \frac{ |\Omega|_g \omega_NR^{\nu}}{ C_{F(\Omega)}}\left(
 \mu_k - \frac{\nu R^{2} }{N+2} \fint_\Omega  w\,dv_g- \fint_\Omega  \widetilde{V} w\,dv_g\right)
\end{equation}
for all $R>0$, or
 \begin{equation*}
\sum_{j=0}^{k-1}  \mu_j  \le  \left( k -\frac{ |\Omega|_g \omega_NR^{\nu}}{ C_{F(\Omega)}}\right)\mu_k  
+\frac{ |\Omega|_g \omega_NR^{\nu}}{ C_{F(\Omega)}}\left( \frac{\nu R^{2} }{N+2} \fint_\Omega  w\,dv_g+ \fint_\Omega  \widetilde{V} w\,dv_g\right).
\end{equation*}
Choosing $R$ such that $\frac{ |\Omega|_g \omega_NR^{\nu}}{ C_{F(\Omega)}}=k$ we get
\begin{equation}\label{Kro_gen5}
\sum_{j=0}^{k-1}  \mu_j  \le k  \left( \frac{\nu}{N+2}\left(\frac{C_{F(\Omega)}k }{|\Omega|_g \omega_N}\right)^{\frac 2\nu} \fint_\Omega  w\,dv_g+ \fint_\Omega  \widetilde{V} w\,dv_g\right),
\end{equation}
which leads to \eqref{Kroger_gen} after taking the infimum with respect to $F$ and $N$. 

\smallskip

The inequality \eqref{heat_gen} is a consequence of \eqref{Lieb-Th_gen} and the following identity 
relating the heat trace to the Laplace transform of the Riesz mean :
\begin{equation}\label{laplace}
\sum_{j\ge 0}{e^{- \mu_j t}} = t^2 \int_0^\infty{e^{-z t} \sum_{j\ge 0}(z-\mu_j)_+}dz.
\end{equation}

\end{proof}

\begin{rems}\label{rem_leg} 
1. In \cite[Theorem 1.2]{LiTa}, Li and Tang obtained for the Laplacian (i.e. in the case $V=\rho=0$, $w=1$) an inequality which is similar to but weaker than \eqref{Kro_gen5}. Indeed, instead of  the term $\frac\nu {N+2}$ in the right side  their inequality appears with  $\frac N {N+2}$.

\smallskip

\noindent 2. It is possible to derive \eqref{Kroger_gen} from \eqref{Lieb-Th_gen} using Legendre transform. Indeed, the Legendre transform of a function $f$ of the form $f(z)=A(z-B)_{+}^{1+\frac \nu 2}$ with $A>0$ is given by
$$f^{\wedge}(  p ) = \sup_{z\ge 0} \left( pz-f(z) \right) = \left(\frac 2A\right)^{\frac 2\nu} \frac {\nu}{(\nu+2)^{{1+\frac2 \nu }} }  p^{1+\frac2 \nu } + Bp,$$
 while the Legendre transform of $ g(z)= \sum_{j\ge 0} \left(z- \mu_j\right)_+ $ is $$g^{\wedge}(  p )=\sum_{j=0 }^{\lfloor p\rfloor-1} \mu_j +(p-\lfloor p\rfloor)\mu_{ \lfloor p\rfloor}.$$
(Indeed, for $z\in [ \mu_{k-1}, \mu_k]$, $pz-g(z) = (p-k)z+\sum_{j=0}^{k-1} \mu_j$ which is nondecreasing as soon as $k\le \lfloor p\rfloor$.) Hence, it suffices to apply Legendre transform to both sides of \eqref{Lieb-Th_gen} taking into account that such a  transform is inequality-reversing. 
\end{rems} 
 \begin{cor}\label{mu_bd_general} Under the assumptions of Theorem \ref{general}, for any positive integer $k$,
\begin{equation} \label{mu_upperbd}
  \widetilde{\mu_k}\leq   \left(1+ 2\sqrt{\frac{1-S_k}{\nu+2}}\right)\left(\frac{ H_\Omega }{|\Omega|_g }k\right)^{\frac 2\nu}\fint_\Omega  w\, dv_g ,
\end{equation}
where 
$$\widetilde{\mu_k} = \mu_k- \fint_\Omega  \widetilde{V} w\,dv_g\quad  \mbox{and}\quad S_k:=\frac{\frac{1}{k}\sum_{j=0}^{k-1} \widetilde{\mu_j}}{ \left(\frac{H_\Omega k}{|\Omega|_g}\right)^{\frac 2 \nu} \fint_\Omega w\,dv_g
}.$$
\end{cor}
Notice that according to Theorem \ref{general} (2), $S_k\le 1$.
\begin{proof} [Proof of Corollary \ref{mu_bd_general}] We  take back the proof of Theorem \ref{general} and rewrite  \eqref{Kro_gen1}  as follows:  For every positive $R$, 
\begin{equation}\label{Kro_gen2}
 k\widetilde\mu_k  -\sum_{j=0}^{k-1}  \widetilde\mu_j   \ge \frac{ |\Omega|_g \omega_NR^{\nu}}{ C_{F(\Omega)}} \left(
 \widetilde\mu_k - \frac{\nu R^{2} }{N+2} \fint_\Omega  w\,dv_g\right),
\end{equation}
which yields
$$\left(  \frac{ |\Omega|_g \omega_NR^{\nu}}{ C_{F(\Omega)}}- k\right) \widetilde\mu_k \le \frac{\nu }{N+2} \frac{ |\Omega|_g \omega_NR^{\nu+2}}{ C_{F(\Omega)}} \fint_\Omega  w\,dv_g -  \sum_{j=0}^{k-1}  \widetilde\mu_j .$$
With the change of variable $\sigma:=\displaystyle \frac{ |\Omega|_g\omega_N}{C_{F(\Omega)} k} R^{\nu}$ the last inequality reads
$$(\sigma-1) \widetilde{\mu_k}\leq   \frac{\nu}{N+2}\left(\frac{C_{F(\Omega)}k }{|\Omega|_g \omega_N}\right)^{\frac 2\nu}\fint_\Omega  w\,dv_g\ \sigma^{1+\frac 2\nu} - \frac 1k \sum_{j=0}^{k-1}  \widetilde\mu_j$$
for all $\sigma >1$.
Taking the infimum with respect to $F$ and $N$ and using \eqref{def-Horm}, we get
$$ (\sigma-1) \widetilde{\mu_k}\leq   \left(\frac{H_\Omega k }{|\Omega|_g }\right)^{\frac 2\nu}\fint_\Omega  w\,dv_g\ \sigma^{1+\frac 2\nu} - \frac 1k \sum_{j=0}^{k-1}  \widetilde\mu_j$$
$$\qquad \qquad = \left(\frac{H_\Omega k }{|\Omega|_g }\right)^{\frac 2\nu}\fint_\Omega  w\,dv_g\left( \sigma^{1+\frac 2\nu} - S_k\right).$$
That is
\begin{equation}\label{Kro_gen3}
\widetilde{\mu_k}\leq   \left(\frac{H_\Omega k }{|\Omega|_g }\right)^{\frac 2\nu}\fint_\Omega  w\,dv_g\frac{\sigma^{1+\frac 2\nu} - S_k }{\sigma-1 }  .
\end{equation}

This inequality can be explicitly optimized 
with respect to $\sigma\in[1,+\infty)$ only when $\nu=2$, and we then
obtain $\sigma_{+}=1+\sqrt{1-S_k}$,
yielding the desired bound.
For general $\nu\geq 2 $ we introduce a new change of variable as follows : $\displaystyle \sigma=1+\alpha z_k$, where $\displaystyle z_k=(1-S_k)^{\frac1{p}}$, $\displaystyle p=\frac{\nu+2}{\nu}$, and $\alpha$ is a free positive parameter. Then the bound \eqref{Kro_gen3} reads:
\begin{equation*}
  \widetilde{\mu_{k}}\leq \, \left(\frac{H_\Omega k }{|\Omega|_g }\right)^{\frac 2\nu}\fint_\Omega  w\,dv_g \ \frac{(1+\alpha z_k)^{p}-1+z_k^p}{\alpha z_k}.
\end{equation*}
Since $1< p\leq 2$ for all $\nu\geq 2$, it follows that
\begin{equation*}
\begin{split}
  \frac1{p}\,\frac{(1+\alpha z_k)^{p}-1+z_k^p}{\alpha z_k}&=\frac{1}{\alpha z_k}\int_0^{\alpha z_k}(1+s)^{p-1}ds+\frac{z_k^{p-1}}{p\alpha}\\
  &\leq\frac{1}{\alpha z_k}\int_0^{\alpha z_k}(1+(p-1)s)\,ds+\frac{z_k^{p-1}}{p\alpha}\\
  &=1+\frac{(p-1)\alpha z_k}{2}+\frac{z_k^{p-1}}{p\alpha}.\\
 \end{split}
\end{equation*}
Optimizing with respect to $\alpha$ leads to choose $\displaystyle \alpha^2=\frac{2z_k^{p-2}}{p(p-1)}$. Thus,
 $$\frac1{p}\ \frac{(1+\alpha z_k)^{p}-1+z_k^p}{\alpha z_k}\le 1 +\sqrt {2}\sqrt {\frac{p-1}p}z_k^{\frac p2} = 1+2 \frac{\sqrt {1-S_k}}{\sqrt {\nu+2}},$$
which implies the the desired inequality.

\end{proof}

\begin{cor}\label{mu_bd_pos} Under the assumptions of Theorem \ref{general}, for any integer $k\in\N$ such that $\sum_{j=0}^{k-1} {\mu_j}\ge 0$,
\begin{equation} \label{mu_upperbd_pos}
  {\mu_k}\left(1- \frac{\fint_\Omega  \widetilde{V} w\,dv_g}{\mu_k}  \right)_+^{1+\frac 2\nu}  \leq     \left(\frac{\nu+2}{2}\right)^{\frac 2\nu} \left(\frac{ H_\Omega }{|\Omega|_g } k\right)^{\frac 2\nu}\fint_\Omega  w\, dv_g .
\end{equation}
In particular,
\begin{equation} \label{mu_upperbd_pos1}
\mu_k\le \max\left\{ 2\fint_\Omega  \widetilde{V} w\,dv_g\ ; \ 2 \left(\nu+2\right)^{\frac 2\nu} \left(\frac{ H_\Omega }{|\Omega|_g } k\right)^{\frac 2\nu}\fint_\Omega  w\, dv_g 
\right\}.
\end{equation}
\end{cor}

\begin{proof} From the inequality  \eqref{Kro_gen1}  in the proof of Theorem \ref{general}, we deduce with  $\sum_{j=0}^{k-1} {\mu_j}\ge 0$
 that for all $R\ge 0$, 
\begin{equation}\label{Kro_gen4}
 k\mu_k   -    \frac{ |\Omega|_g \omega_NR^{\nu}}{ C_{F(\Omega)}} \left(
 \mu_k - \frac{\nu R^{2} }{N+2} \fint_\Omega  w\,dv_g - \fint_\Omega \widetilde{V} w\,dv_g \right)\ge 0.
\end{equation}
The left side achieves its minimum when $R=0$ if $\mu_k\le \fint_\Omega \widetilde{V} w\,dv_g$ and otherwise when $R^2 =  \frac{N+2}{\nu+2} \left(
 \mu_k - \fint_\Omega \widetilde{V} w\,dv_g \right)/\fint_\Omega  w\,dv_g$. Since \eqref{mu_upperbd_pos} is obviously satisfied when $\mu_k\le \fint_\Omega \widetilde{V} w\,dv_g$, we shall assume $\mu_k> \fint_\Omega \widetilde{V} w\,dv_g$ and get
 $$k\mu_k   -    \frac{ |\Omega|_g \omega_N}{ C_{F(\Omega)}} \left( \frac{N+2}{\nu+2}\right)^{\frac \nu2}\frac{2}{\nu+2} \left(
 \mu_k - \fint_\Omega \widetilde{V} w\,dv_g \right)^{1+\frac\nu2}\left(\fint_\Omega  w\, dv_g\right)^{-\frac \nu2}\ge 0$$
which gives
$$ \left(
 \mu_k - \fint_\Omega \widetilde{V} w\,dv_g \right)^{1+\frac\nu2}\le \frac{ C_{F(\Omega)}} { |\Omega|_g \omega_N}\left( \frac{\nu+2}{N+2}\right)^{\frac \nu2}\frac{\nu+2}2 \left(\fint_\Omega  w\, dv_g\right)^{\frac \nu2} k\mu_k .$$
 Therefore,
 $$  \mu_k^{\frac \nu2}\left(1
- \frac{\fint_\Omega \widetilde{V} w\,dv_g}{\mu_k} \right)^{1+\frac\nu2}\le \frac{ C_{F(\Omega)}} { |\Omega|_g \omega_N}\left( \frac{\nu+2}{N+2}\right)^{\frac \nu2}\frac{\nu+2}2 \left(\fint_\Omega  w\, dv_g\right)^{\frac \nu2} k.$$
Raising to the power $\frac2\nu$ and taking the infimum with respect to $F$ and $N$  we obtain  \eqref{mu_upperbd_pos}.

To prove \eqref{mu_upperbd_pos1} we observe that if $\mu_k>2\fint_\Omega  \widetilde{V} w\,dv_g$, then $1- \frac{\fint_\Omega  \widetilde{V} w\,dv_g}{\mu_k} >\frac 12$, so  we can deduce from \eqref{mu_upperbd_pos} 
$$\left(\frac 12\right)^{1+\frac2\nu} \mu_k\leq     \left(\frac{\nu+2}{2}\right)^{\frac 2\nu} \left(\frac{ H_\Omega }{|\Omega|_g } k\right)^{\frac 2\nu}\fint_\Omega  w\, dv_g .
$$

\end{proof}

\smallskip

Note that when $\rho=V=0$, the inequality \eqref{mu_upperbd_pos} of Corollary \ref{mu_bd_pos} produces
$$
  {\mu_k}  \leq     \left(\frac{\nu+2}{2}\right)^{\frac 2\nu} \left(\frac{ H_\Omega }{|\Omega|_g } k\right)^{\frac 2\nu}\fint_\Omega  w\, dv_g,
$$
which  coincides with Kröger's estimate \cite[Corollary 2]{Kro} when $\Omega$ is a Euclidean domain and $w=1$ (just replace $H_\Omega$ by $\frac{(2\pi)^\nu}{\omega_\nu}$).

\smallskip

Let us highlight some consequences of Theorem  \ref {general} on Schrödinger operators, Witten Laplacians,  and Laplacians associated with conformally Euclidean metrics. 

\begin{ex}[Schrödinger operators]\label{Schrodinger}
 From  \eqref{Kroger_gen} in Theorem \ref{general} and \eqref{mu_upperbd_pos1} in Corollary \ref{mu_bd_pos} we deduce  that  for any Schrödinger operator $\Delta_g+V$ on $\Omega$ and any  integer $k\ge 0$ we find (with $\rho=0$ and $w=1$)
\begin{equation} \label{Kro_Schrod}
\frac{1}{k} \sum_{j=0}^{k-1} \mu_j (\Delta_g+V)\le \frac{\nu}{\nu+2}   \left(\frac{H_\Omega }{|\Omega|_g}k\right)^{\frac 2 \nu} 
+\fint_\Omega  {V} \,dv_g.
\end{equation}
Furthermore, if  $\sum_{j=0}^{k-1} {\mu_j}(\Delta_g+V)\ge 0$,
\begin{equation} \label{Schrod_upperbd}
\mu_k(\Delta_g+V)\le \max\left\{ 2\fint_\Omega  {V} \,dv_g\ ; \ 2 \left(\nu+2\right)^{\frac 2\nu} \left(\frac{ H_\Omega }{|\Omega|_g } k\right)^{\frac 2\nu}
\right\}.
\end{equation}
These estimates are to be compared with \cite[Theorem 2.2 and Corollary 2.8]{EI1}, \cite[Theorem 2.1]{EI2}, and the results by Grigor'yan, Netrusov and Yau \cite[Theorem 5.15 and (1.14)]{GNY} by which, under the assumption  ${\mu_0}(\Delta_g+V) \ge0$,
\begin{equation} \label{Schrod_GNY}
\mu_k(\Delta_g+V)\le C(\Omega) k+  \frac1{\varepsilon(\Omega)} \fint_\Omega  {V} \,dv_g
\end{equation}
where $C(\Omega) >0$ and $\varepsilon(\Omega)  \in(0,1)$ are two Riemannian constants that do not depend on $V$ or $k$.  They ask whether such an estimate holds true with  $\varepsilon(\Omega) =1$. 
The inequality \eqref{Kro_Schrod} answers this question for the eigenvalue sums $ \sum_{j=0}^{k-1} \mu_j $ in the affirmative, without any positivity condition.  

\smallskip
On the other hand, unlike the upper bound in \eqref{Schrod_GNY},   our estimates  \eqref{Kro_Schrod} and \eqref{Schrod_upperbd}  are  consistent with the Weyl law   regarding the power of $k$. Notice that  \eqref{Schrod_GNY} has been recently improved by A. Hassannezhad \cite{Has} who obtained
$$ 
\mu_k(\Delta_g+V)\le  C(\Omega) + A_\nu  \fint_\Omega  {V} \,dv_g + B_\nu \left(\frac{ k }{|\Omega|_g } \right)^{\frac 2\nu}$$
under  the same assumption of positivity of  ${\mu_0}(\Delta_g+V)$,
where  $A_\nu>1$ and $B_\nu$ are two  constants  that only depend on the dimension $\nu$, and  $C(\Omega)$ is a Riemannian constant that does not depend on $V$ or $k$.
Our estimates are valid, however, under weaker assumptions and, moreover, the coefficient in front of $ \fint_\Omega  {V} \,dv_g$ in \eqref{Schrod_upperbd} is equal to 1 while the  other coefficient is explicitly computable at least in the elementary case where $\Omega$ is conformally Euclidean.

\end{ex}

\begin{ex}[Witten Laplacians]\label{Witten} Let $\Omega$ be a bounded domain of a Riemannian manifold $(M,g)$ and let $\Delta_\rho$ be the Witten Laplacian associated with the density $e^{-2\rho}$, that is,
$$\Delta_\rho \varphi = \Delta_g \varphi + 2(\nabla^g\rho, \nabla^g\varphi) .$$
The Neumann eigenvalues $\{\mu_l\}$ of $\Delta_\rho$ in $\Omega $ satisfy the following estimates :

\noindent (1) For all $z\in\R$,
\begin{equation}\label{Lieb-Th_Witten}
 \sum_{j\ge 0} \left(z- \mu_j\right)_+ \ge \frac{2\  \vert\Omega\vert_g}{(\nu+2) H_\Omega} \left( z- \fint_\Omega \vert\nabla^g\rho\vert^2 \, dv_g    \right)_+^{1+\frac\nu 2}.
\end{equation}

\smallskip

\noindent (2) For all $k\in\N^*$,
\begin{equation}\label{Kroger_Witten}
\frac{1}{k} \sum_{j=0}^{k-1} \mu_j \le \frac{\nu}{\nu+2}   \left(\frac{H_\Omega }{|\Omega|_g}k\right)^{\frac 2 \nu} 
+\fint_\Omega  \vert\nabla^g\rho\vert^2\,dv_g .   
\end{equation}

\noindent (3) For all $k\in\N^*$,
\begin{equation} \label{mu_upperbd_pos_Wit}
  {\mu_k}\left(1- \frac{\fint_\Omega  \vert\nabla^g\rho\vert^2\,dv_g}{\mu_k}  \right)_+^{1+\frac 2\nu}  \leq     \left(\frac{\nu+2}{2}\right)^{\frac 2\nu} \left(\frac{ H_\Omega }{|\Omega|_g } k\right)^{\frac 2\nu}.
\end{equation}
In particular,
\begin{equation} \label{mu_upperbd_pos1_Wit}
\mu_k\le \max\left\{ 2\fint_\Omega  \vert\nabla^g\rho\vert^2\,dv_g\ ; \ 2 \left(\nu+2\right)^{\frac 2\nu} \left(\frac{ H_\Omega }{|\Omega|_g } k\right)^{\frac 2\nu}
\right\}.
\end{equation}
This last inequality  is to be compared with the estimates obtained in \cite{CES}.

 \smallskip
 For example, when $\Omega$ is a bounded domain of $\R^\nu$ endowed with the Gaussian density $e^{-\vert x\vert^2/2}$, we have for the corresponding Witten Laplacian
$$ \frac{1}{k} \sum_{j=0}^{k-1} \mu_j \le \frac{\nu}{\nu+2}  \left( 4\pi^2\left(\frac{ k}{|\Omega| \omega_\nu}\right)^{\frac 2 \nu} 
+\frac{ \omega_\nu}{|\Omega|} R^{\nu+2} \right),
 $$
 where $R$ is chosen so that $\Omega$ is contained in the Euclidean ball  $B_R$.

 \end{ex}

 \begin{ex}[Laplacian associated with a conformally Euclidean metric]\label{conf} Let $\Omega$ be a bounded domain of $\R^\nu$  and let $ g=\alpha^{-2} g_E$ be a Riemannian metric that is conformal to the Euclidean metric $g_E$. The Neumann eigenvalues $\{\mu_l\}$ of the Laplacian $\Delta_{ g}$ in $\Omega$ 
 satisfy the following estimates in which $\vert\Omega\vert$ denotes the Euclidean volume of $\Omega$ :

\smallskip
\noindent (1) For all $z\in\R$,
\begin{equation}\label{Lieb-Th_conf}
 \sum_{j\ge 0} \left(z- \mu_j\right)_+ \ge \frac{2 \omega_\nu \vert\Omega\vert}{(\nu+2) (2\pi)^\nu} \left( \fint_\Omega \alpha^{2} \, d^\nu x\right)^{\frac\nu 2} \left( z- \frac{\nu^2}{4}\fint_\Omega \vert\nabla\alpha\vert^2  \, d^\nu x    \right)_+^{1+\frac\nu 2}.
\end{equation}

\smallskip

\noindent (2) For all $k\in\N$,
\begin{equation}\label{Kroger_conf}
\frac{1}{k} \sum_{j=0}^{k-1} \mu_j \le \frac{4\pi^2 \nu}{\nu+2}   \left(\frac{ k}{\omega_\nu |\Omega|}\right)^{\frac 2 \nu} \fint_\Omega \alpha^{2} \, d^\nu x 
+\frac{\nu^2}{4}\fint_\Omega \vert\nabla\alpha\vert^2 \, d^\nu x.   
\end{equation}

\noindent (3) For all $k\in\N$,
\begin{equation} \label{mu_upperbd_pos_conf}
  {\mu_k}\left(1- \frac{\nu^2}{4}\frac{\fint_\Omega \vert\nabla\alpha\vert^2 \, d^\nu x}{\mu_k}  \right)_+^{1+\frac 2\nu}  \leq    4\pi^2 \left(\frac{\nu+2}2\right)^{\frac 2\nu} \left(\frac{ k }{\omega_\nu |\Omega|} \right)^{\frac 2\nu}\fint_\Omega \alpha^{2} \, d^\nu x.
\end{equation}
In particular,
\begin{equation} \label{mu_upperbd_pos1_conf}
\mu_k\le \max\left\{ \frac{\nu^2}{2}\fint_\Omega \vert\nabla\alpha\vert^2 \, d^\nu x\ ; \ 8\pi^2 \left(\nu+2\right)^{\frac 2\nu} \left(\frac{ k }{\omega_\nu |\Omega| } \right)^{\frac 2\nu}\fint_\Omega \alpha^{2} \, d^\nu x
\right\}.
\end{equation}

Note that a domain  of the hyperbolic space $\bf H^\nu$ can be identified with a domain of the Euclidean unit ball endowed with the metric $g=\left(\frac{2}{1-\vert x\vert^2}\right)^2g_E$.  For such a domain we get, with $\alpha= \frac{1-\vert x\vert^2}{2}$,  $\fint_\Omega \alpha^{2} \, d^\nu x\le \frac 14$, and $\fint_\Omega \vert\nabla\alpha\vert^2 \, d^\nu x = \fint_\Omega \vert x\vert^2 \, d^\nu x$, 

\begin{equation}\label{Kroger_hyp}
\frac{1}{k} \sum_{j=0}^{k-1} \mu_j \le \frac{ \pi^2 \nu}{\nu+2}  \left(\frac{ k}{\omega_\nu |\Omega|}\right)^{\frac 2 \nu} +\frac{\nu^2}{4} \fint_\Omega \vert x\vert^2 \, d^\nu x.   
\end{equation}

\begin{equation} \label{mu_upperbd_pos_hyp}
  {\mu_k}\left(1- \frac{\nu^2}{4}\frac{ \fint_\Omega \vert x\vert^2 \, d^\nu x}{\mu_k}  \right)_+^{1+\frac 2\nu}  \leq    \pi^2\left(\frac{\nu+2}2\right)^{\frac 2\nu} \left(\frac{ k }{\omega_\nu |\Omega| } \right)^{\frac 2\nu}
\end{equation}
and
\begin{equation} \label{mu_upperbd_pos1_hyp}
\mu_k\le \max\left\{ \frac{\nu^2}{2}\fint_\Omega \vert x\vert^2\, d^\nu x\ ; \ 2\pi^2 \left(\nu+2\right)^{\frac 2\nu} \left(\frac{ k }{\omega_\nu |\Omega| } \right)^{\frac 2\nu}
\right\}.
\end{equation}
\end{ex}


\section{Sums of Neumann eigenvalues on domains conformal to Euclidean sets,
and phase-space volumes}\label{phase-space-sect}

A phase-space analysis can considerably sharpen the upper bounds on sums
of eigenvalues from the previous sections so that they become sharp in the semiclassical regime.  
Following physical tradition, it is shown in \cite{LiLo} how this may be achieved 
in some circumstances with the aid of coherent 
states.  We carry out such an analysis in this section 
for \eqref{basicsetup} when $(M,g) = (\R^\nu, d^{\nu}x)$.
We must
first introduce a few quantities that will be helpful to relate spectral estimates to phase-space
volumes.  To avoid complications we assume that the potential energy 
$V$ is Lipschitz continuous and bounded from below.
We do not assume that $\Omega$ is necessarily bounded, but if it is not, we require $V$
to be {\em confining} in the sense that there is a radial function $V_{\rm rad}(r)$ tending to $+\infty$ 
as $r \to \infty$ with $V({\bf x}) \ge V_{\rm rad}(|{\bf x}|)$ for all ${\bf x} \notin \Omega$.  This condition 
is sufficient to ensure that the eigenvalues form a discrete sequence tending to $+\infty$.

\begin{defi}
The {\em effective potential} incorporating a correction for the
conformal transformation will be denoted
$\widetilde{V}({\bf x}) := V({\bf x}) + |\nabla \rho|^2({\bf x})$, and
the maximal Lipschitz constant of $\widetilde{V}({\bf x})$
on the region
$\Omega \cap \{{\bf x}: \widetilde{V}({\bf x}) \le \Lambda\}$
will be denoted ${\rm Lip}(\Lambda)$.

The $L^2$-normalized ground-state Dirichlet eigenfunction for the ball of geodesic radius $r$
in $M$
will be denoted $h_r$ and
$\mathcal{K}(h_r) := \int_{B_r}{|\nabla h_r({\bf x})|^2 d^\nu\,x}$.
I.e., in this section where $M = \R^\nu$, $h$ is a scaled Bessel function and
$$
\mathcal{K}(h_r) = \frac{j_{\frac{\nu}{2}-1, 1}^2}{r^2}.
$$
\end{defi}

\begin{rem}
The function $h_r$ will ensure that some coherent-state functions
to be defined below are localized in configuration space.  Its specific form is but one
of many plausible choices.
\end{rem}

We next recall some quantities that arise in phase-space analysis.

\begin{defi}
The {\em Euclidean phase-space volume} for energy $\Lambda$
is defined as
$$
\Phi_1(\Lambda) := \frac{1}{(2 \pi)^\nu}
\Big|({\bf x},{\bf p}): |{\bf p}|^2 + \widetilde{V}({\bf x}) \le \Lambda \Big| =
\frac{\omega_\nu}{(2 \pi)^\nu} \int_{\Omega}{\left(\Lambda -
\widetilde{V}({\bf x})\right)_+^{\frac{\nu}{2}} d^\nu x},
$$
according to a standard
calculation to be found, for example, in \cite{LiLo}. 
If the weight in \eqref{basicsetup} is not constant, we
make use of a
{\em weighted phase-space volume},
$$
\Phi_w(\Lambda) = \frac{\omega_\nu}{(2 \pi)^\nu} \int_{\Omega}{\left(\Lambda - \widetilde{V}({\bf x})\right)_+^{\frac{\nu}{2}}
w({\bf x}) d^\nu x}.$$
The total energy associated with this quantity is correspondingly
\begin{align}
E_w(\Lambda) &:= \frac{1}{(2 \pi)^\nu} \int_{\{({\bf x},{\bf p}): {\bf x} \in \Omega, |{\bf p}|^2 + \widetilde{V}({\bf x}) \le \Lambda\}}
{\left(|{\bf p}|^2 + \widetilde{V}({\bf x})\right) w({\bf x}) d^\nu x} d^\nu p\nonumber\\
&= \frac{\nu}{\nu+2} \frac{\omega_\nu}{(2 \pi)^\nu} \int_{\Omega}
{\left(\Lambda - \widetilde{V}({\bf x})\right)_+^{1+\frac{\nu}{2}} w({\bf x}) d^\nu x}.\label{EvsPhi}
\end{align}
\end{defi}

\noindent
We note that according to \eqref{EvsPhi},
\begin{equation}\label{Ecvx}
\frac{d E_w}{d \Lambda}\left(\Lambda\right) = \Phi_w(\Lambda),
\end{equation}
and that $\Phi_w$
increases strictly monotonically in $\Lambda$, implying that $E_w$ is strictly convex.

\begin{thm}\label{PSBound}
Let $\mu_0 \le \mu_1 \le \dots$ be the variationally defined Neumann eigenvalues \eqref{NeumannDef} 
on an open set $\Omega \in \R^\nu$, where
$w, \rho$, and $V$ satisfy the assumptions stated above,
and define $\Lambda(k)$ as the minimal value
of $\Lambda$ for which
$\Phi_1(\Lambda) \ge (2 \pi)^\nu k$.  Then
\begin{equation}\label{EucUpperbd}
\sum_{j=0}^{k-1} \mu_j \le
E_w(\Lambda(k))
+ 3 \left( 2 j_{\nu-1,1}^2 {\rm Lip}(\Lambda(k))\right)^{\frac 1 3} \Phi_w\left(\Lambda(k) + (2 j_{\nu-1,1}^2 {\rm Lip}(\Lambda(k))^{\frac 1 3}\right).
\end{equation}
{
The Riesz-mean form of the inequality reads}
\begin{align}\label{EucRieszbd}
\sum_{j=0}^{k-1} \left(z - \mu_j\right)_+  \ge \,\,
& \frac{2}{\nu+2} \frac{\omega_\nu}{(2 \pi)^\nu} 
\int{\left(z - V(y)\right)_+^{1+\frac{\nu}{2}} dy}\\
&-\left( \frac{\omega_\nu}{(2 \pi)^\nu} 
\int{\left(z - V(y)\right)_+^{\frac{\nu}{2}} dy}\right) \left(\|\nabla h_r\|^2 + \int{|{\bf x}|h_r^2}\right).
\end{align}
\end{thm}

\begin{rems}
1.  We call attention to the fact that the condition in this theorem
defining $\Lambda$ uses the Euclidean phase space, whereas weighted phase-space quantities
appear in \eqref{EucUpperbd}.

\smallskip
\noindent 2.  The dominant term in the semiclassical regime can be identified by introducing a small parameter $\alpha$ as a coefficient of $|\nabla \varphi|^2$ in \eqref{basicsetup}, i.e.,
\begin{equation}\label{semisetup}
\mathcal{R}_{\alpha}(\varphi) := \frac{\int_\Omega{(\alpha |\nabla \varphi({\bf x})|^2 + V({\bf x})|\varphi({\bf x})|^2 )}
w({\bf x}) e^{-2 \rho({\bf x})} dv_g}{\int_\Omega{|\varphi({\bf x})|^2 e^{-2 \rho({\bf x})} dv_g}}.
\end{equation}
{The result, in the Riesz-mean form  after choosing a convenient 
relationship between $r$ and $\alpha$, is}
\begin{align}\label{EucRieszbdalpha}
\sum_{j=0}^{k-1} \left(z - \mu_j\right)_+  \ge \,\,
&\alpha^{-\frac{\nu}{2}} \frac{2}{\nu+2} \frac{\omega_\nu}{(2 \pi)^\nu} 
\int{\left(z - V(y)\right)_+^{1+\frac{\nu}{2}} dy}\nonumber\\
&-\left(\alpha^{\frac 1 3 -\frac{\nu}{2}}  \frac{\omega_\nu}{(2 \pi)^\nu} 
\int{\left(z - V(y)\right)_+^{\frac{\nu}{2}} dy}\right) \left(\|\nabla h_1\|^2 + \int{|{\bf x}|h_1^2}\right),
\end{align}
{in which the leading term is precisely the expected semiclassical expression,
in contrast to results of the previous section such as Example \ref{conf}.}

\smallskip
\noindent  3.  Inequalities of the type \eqref{EucUpperbd} imply estimates of
quantities including trace of the heat kernel (= the partition function in 
quantum physics)
and the spectral zeta functions by simple transforms.  For instance
\eqref{EucUpperbd} implies for the
Riesz mean $R_1(z) := \sum_{j}{(z - \mu_j)_{+}}$ that
\begin{equation}\label{RieszEucbd}
R_1(z) \ge \frac{1}{(2 \pi)^\nu} \left(z \Phi_w(\Lambda(z)) -
E_w(\Lambda(k))
+ 3 \left( 2 j_{\nu-1,1}^2 {\rm Lip}(\Lambda(k))\right)^{\frac 1 3} \Phi_w\left(\Lambda(k) + (2 j_{\nu-1,1}^2 {\rm Lip}(\Lambda(k))^{\frac 1 3}\right)\right).
\end{equation}
\normalsize
The Riesz mean is in turn related to the heat trace by the Laplace transform \eqref{laplace}.
\end{rems}

\begin{proof}[Proof of Theorem \ref{PSBound}]
We apply Theorem \ref{AvePrin} to the Neumann eigenvalues of
\eqref{basicsetup} as defined by
\eqref{NeumannDef}, using for test functions
``coherent states'' \cite{LiLo,Th3} of the form:
$$
f_\zeta({\bf x}) := \frac{1}{(2 \pi)^{\nu/2}} e^{i {\bf p} \cdot ({\bf x}) + \rho({\bf x})} h_r({\bf x}- {\bf y}).
$$
In this formula,
$\zeta = ({\bf p},{\bf y})$ ranges over
the phase space
$\mathfrak{M} = \R^{2 \nu}$ with Lebesgue measure.
The radius $r$ will be chosen below.

\smallskip
We note that the inner product that appears is a Fourier transform
with respect to the variable ${\bf x}$, {\em viz}.,
$$
\left\langle\phi, f_\zeta\right\rangle = \mathfrak{F}[h_r({\bf x} - {\bf y})e^{-\rho({\bf x})} \phi({\bf x})],
$$
where if $\Omega$ is a strict subset of $\R^\nu$, then $\phi$ is extended by $0$
outside $\Omega$.
Thus, with the Parseval identity,
\begin{align}\label{weightedParseval}
\int_{\R^{2 \nu}}{\left|\left\langle\phi, f_\zeta\right\rangle\right|^2 d^{\nu}p\,d^{\nu}y}
=
\int_{\R^{\nu}}\int_{\R^{\nu}}{h_r({\bf x} - {\bf y})^2 \left|\phi\right|^2 e^{-2 \rho} d^{\nu}y\, d^{\nu}x}
=\int_{\R^{\nu}}{\left|\phi\right|^2 e^{-2 \rho} d^{\nu}x}= \|\phi\|^2.
\end{align}
The set $\mathfrak{M}_0$ in Theorem \ref{AvePrin}
must be taken large enough so that
\begin{align}\label{EucCond}
k \le \int_{\mathfrak{M}_0}{\|f_\zeta\|^2 d \sigma}
= \frac{1}{(2 \pi)^\nu} \int_{\mathfrak{M}_0}{\int_{\Omega}{h_r^2({\bf x}- {\bf y}) e^{2 \rho({\bf x}) - 2 \rho({\bf x})}d^{\nu}xd^{\nu}p\,d^{\nu}y}} \le \frac{1}{(2 \pi)^\nu} |\mathfrak{M}_0|,
\end{align}
in which case
$$
\sum_{j=0}^{k-1} \mu_j \le \int_{\mathfrak{M}_0}{\mathcal{R}(f_\zeta) d^{\nu}p\,d^{\nu}y}= \quad\quad\quad\quad\quad\quad\quad\quad\quad\quad
\quad\quad\quad\quad\quad\quad\quad\quad\quad\quad\quad\quad\quad\quad\quad\quad\quad\quad
$$
$$
=\frac{1}{(2 \pi)^\nu} \int_{{\mathfrak{M}_0} \times \Omega}{\left((|{\bf p}|^2 +V({\bf x}))h_r^2({\bf x}-{\bf y}) + |\nabla h_r({\bf x}-{\bf y}) + h_r({\bf x}-{\bf y}) \nabla \rho({\bf x})|^2
\right) w({\bf x}) d^{\nu}x\,d^{\nu}p\,d^{\nu}y}.\nonumber \\
$$
\normalsize
\begin{equation}\label{ubcalc}
\,
\end{equation}
We now make the ansatz that $\mathfrak{M}_0 = \mathfrak{M}_0(\Lambda) := \{({\bf x},{\bf p}):\, {\bf x} \in \Omega, |{\bf p}|^2 + \tilde{V}({\bf x}) \le \Lambda\}$,
where $\Lambda \ge \Lambda(k)$, defined
as the minimum value of $\Lambda$
for \eqref{EucCond} to be valid.  Thus the upper bound in \eqref{EucCond} becomes
$\frac{1}{(2 \pi)^\nu} \Phi_1(\Lambda)$, whence the condition in the theorem.

Since the support of $h$ is restricted to a ball of radius $r$, the $x$-integral may be restricted to the set
$\{{\bf x} \in \Omega: \exists {\bf y}, |{\bf y}-{\bf x}| \le r, |{\bf p}|^2 + \tilde{V}({\bf x}) \le \Lambda\} \subset \mathfrak{M}_0(\Lambda + {\rm Lip}(\Lambda) r)$.
Thus, integrating first in ${\bf y}$,
the right side of \eqref{ubcalc} is bounded above by
\begin{align}
\frac{1}{(2 \pi)^\nu} & \int_{\{{\bf p}: |{\bf p}|^2 \le \Lambda\}}
\int_{\{{\bf x}: \widetilde{V}({\bf x}) \le \Lambda + {\rm Lip}(\Lambda) r - |{\bf p}|^2\}}
\int_{\R^\nu}
{\left((|{\bf p}|^2 +V({\bf x}))
h_r^2({\bf x}-{\bf y})\right.} \nonumber \\
&\quad\quad\quad\quad+ {\left.|\nabla h_r({\bf x}-{\bf y}) + h_r({\bf x}-{\bf y}) \nabla \rho({\bf x})|^2
\right) w({\bf x}) d^{\nu}y\,d^{\nu}x\,d^{\nu}p}\nonumber\\
&= \frac{1}{(2 \pi)^\nu} \int_{\{{\bf p}: |{\bf p}|^2 \le \Lambda\}}
\int_{\{{\bf x}: \widetilde{V}({\bf x}) \le \Lambda + {\rm Lip}(\Lambda) r - |{\bf p}|^2\}}
\int_{\R^\nu}
{\left((|{\bf p}|^2 +\widetilde{V}({\bf x}))
h_r^2({\bf x}-{\bf y})\right.}  \nonumber \\
&\quad\quad\quad\quad+ \left.|\nabla h_r({\bf x}-{\bf y})|^2 +\nabla \rho({\bf x}) \cdot \nabla h_r^2({\bf x}-{\bf y})
\right) w({\bf x}) d^{\nu}y\,d^{\nu}x\,d^{\nu}p.\nonumber
\end{align}
\normalsize
The last contribution vanishes because
\begin{equation}\label{cancel}
\int_{\R^\nu}{\nabla \rho({\bf x}) \cdot \nabla h_r^2({\bf x}-{\bf y})d^\nu y} =\int_{\R^\nu}{\nabla \rho({\bf x}) \cdot \nabla  1 d^\nu y} = 0,
\end{equation}
leaving
\begin{equation}\label{Eucl bd}
\sum_{j=0}^{k-1} \mu_j \le
\frac{1}{(2 \pi)^\nu} \int_{\mathfrak{M}_0(\Lambda+ {\rm Lip}(\Lambda) r)}
{(|{\bf p}|^2 + \tilde{V}({\bf x})
 +  \mathcal{K}(h_r)) w({\bf x})  d^{\nu}x\,d^{\nu}p}
\end{equation}
for all values of $r>0$.  The upper bound \eqref{Eucl bd}
is of the form
\begin{align}
\frac{1}{(2 \pi)^\nu} &\left(E_w(\Lambda + {\rm Lip}(\Lambda) r)
+ \Phi_w(\Lambda + {\rm Lip}(\Lambda) r) \mathcal{K}(h_r)\right)\nonumber \\
 &\le \frac{1}{(2 \pi)^\nu} \left(E_w(\Lambda)
+ \left(\frac{j_{\nu-1,1}^2}{r^2} + {\rm Lip}(\Lambda) r \right)\Phi_w(\Lambda + {\rm Lip}(\Lambda) r)\right),\nonumber
\end{align}
where we have made use of \eqref{Ecvx} and the monotonicity of $\Phi_w$
in a first-order expansion of $E_w$.  Choosing the optimal
value $r = \left(\frac{2 j_{\nu-1,1}^2}{L}\right)^{\frac 1 3}$, we get the 
{claim \eqref{EucUpperbd}}.

{The derivation of \eqref{EucRieszbd} proceeds similarly.}
\end{proof}

\begin{rem}
We note the following special cases of particular interest.
\item 1.  Laplace operators with Neumann conditions on a compact Euclidean domain ($V=\rho = 0$, $w=1$).  In this case ${\rm Lip}(\Lambda) = 0$, 
$\displaystyle \Lambda^{\frac{\nu}{2}} :=\frac{(2 \pi)^{\nu}}{B_{\nu}}\frac{k}{|\Omega|}$, and
we recover the inequality of Kr\"oger, that
$$
\sum_{j=0}^{k-1} \mu_j \le \frac{\nu}{\nu+2} \frac{\omega_\nu}{(2 \pi)^\nu} |\Omega| \Lambda^{1+\frac{\nu}{2}}=\frac{\nu}{\nu+2}(2\pi)^2\omega_\nu^{-\frac{2}{\nu}}\,\frac{k^{\frac{\nu+2}{\nu}}}{|\Omega|^{\frac{2}{\nu}}}.
$$
Indeed, without the potential $V$, the introduction of the function $h_r$ is not needed for the proof.

\smallskip
\item 2.  Nonhomogeneous   problems with $\rho = V = 0$, but $w$ is variable, under Neumann conditions:
$$
\sum_{j=0}^{k-1} \mu_j \le \frac{\nu}{\nu+2} \frac{\omega_\nu}{(2 \pi)^\nu} \left(\int_{\Omega}{w({\bf x}) d^\nu x}\right) \Lambda^{1+\frac{\nu}{2}}.
$$
The eigenvalue bounds of Corollary \ref{mu_bd_general} are sharp as $k$ tends to infinity. Indeed since $\widetilde{\mu_k}= \mu_k$ we get
$$
\mu_k\leq 4 \pi^2 {\fint_\Omega{w({\bf x}) d^{\nu}x}} \left(\frac{k}{|\Omega| \omega_\nu}\right)^{\frac 2 \nu}
  \bigg(1+ 2\sqrt{\frac{1-S_k}{\nu+2}}\bigg)
$$
with $S_k$ as given in Corollary \ref{mu_bd_general} :
\begin{equation*}
  S_k=\frac{\frac{\nu+2}{\nu}\frac{1}{k}\sum_{j=0}^{k-1} \mu_j}{4 \pi^2 {\fint_\Omega{w({\bf x}) d^{\nu}x}}\left(\frac{k}{|\Omega| \omega_\nu}\right)^{\frac 2 \nu}}\leq 1.
\end{equation*}
\end{rem}


\section{Bounds for Neumann  eigenvalues on subdomains of compact homogeneous spaces}

In this section, we deal with the case where the ambient  space is a compact homogeneous
Riemannian manifold $(M,g)$ with isomorphism group denoted $G$.
In particular, we shall
recover Strichartz's result \cite{Str} 
with a more efficient proof and extend it to a wider class of operators.  We begin with  bounds 
in the spirit of Theorem \ref{general} and then derive a phase-space bound analogous to Theorem 
\ref{PSBound}.

\smallskip

Let us denote by 
$$\mbox{spec}(M) =\left\{0 = \lambda_0< \lambda_1 \le \lambda_2 \le\cdots \le \lambda_k \le \cdots\right\}$$
the spectrum of the Laplace-Beltrami operator $\Delta_g$ on $M$ (each eigenvalue is repeated according to its multiplicity).
 Although $0$ is a simple eigenvalue, all the other eigenvalues are degenerate owing to the transitive action of the isometry group $G$ (recall that the eigenspaces are invariant under the action of $G$.)

Given a regular domain $\Omega\subset M$ endowed with  densities $e^{-2\rho}$ and $e^{-2\theta}=w e^{-2\rho}$, and a potential $V$, we consider the eigenvalues $\mu_l(\Omega, g, \rho,w,V)$, $l\in\N$, defined by \eqref{basicsetup} and \eqref{NeumannDef} and seek for relationships between the $\mu_l$'s and the  $\lambda_l$'s . As before, we will use the notation $\widetilde{V} = V+ \vert\nabla^g\rho\vert^2$. We also need the following subspaces introduced in Theorem \ref{AvePrin} 
$$E_0( R )=\bigoplus_{\mu< R}\ker(H-\mu I)\ \mbox{and }\ E( R )=\bigoplus_{\mu\le R}\ker(H-\mu I)$$
where $H=H(\Omega, g, \rho,w,V)$ is the operator defined by \eqref{Operatror1}. The corresponding  subspaces associated with the Laplacian $\Delta_g$ on $M$ will be denoted 
$$F_0( R )=\bigoplus_{\lambda< R}\ker(\Delta_g-\lambda I)\ \mbox{and }\ F( R )=\bigoplus_{\lambda\le R}\ker(\Delta_g-\lambda I).$$

\begin{thm}\label{gen-homog}
Let $(M,g)$ be a compact homogeneous
Riemannian manifold.
Let $\mu_l=\mu_l(\Omega, g, \rho,w,V)$, $l\in\N$, be the    eigenvalues  defined by \eqref{NeumannDef}
on a bounded open set $\Omega \subset M$.
Then,  for all $z\in\R$,
\begin{equation}\label{Lieb-Th_homog}
 \sum_{j\ge 0} \left(z- \mu_j\right)_+ \ge 
 \frac{ \vert\Omega\vert_g}{\vert M\vert_g}  \sum_{j\ge 0}  \left( z- \tilde\lambda_j   \right)_+,
\end{equation}
where $ \tilde\lambda_j  = \lambda_j \fint_\Omega w \, dv_g +\fint_\Omega \widetilde{V}w \, dv_g  $.
Equality holds in \eqref{Lieb-Th_homog} for some $z\in\R$ if and only if $$E_0(z)\subset e^\rho F(\tilde z)\subset E(z),$$ 
with $ \tilde z  =\frac 1{ \fint_\Omega w \, dv_g } \left(z- \fint_\Omega \widetilde{V}w \, dv_g \right) $. 

\end{thm}

\begin{proof}
Let $\{ y_\lambda \ : \ \lambda\in \mbox{spec}(M) \}$ be an orthonormal basis of $L^2(M,g)$ with $\Delta_g y_\lambda = \lambda y_\lambda$.  The proof relies on Theorem \ref{AvePrin} in which we take   ${\mathfrak M} =\mbox{spec}(M)$ endowed with the uniform discrete measure, and use  test functions of the form 
$$f_\lambda =  y_\lambda e^\rho.$$ 

For any function $\psi\in L^2(\Omega, e^{-2\rho}dv_g)$ (endowed with the norm $\Vert \psi\Vert^2= \int_\Omega \psi^2 e^{-2\rho } dv_g$), 
$$\int_{\mathfrak M}\langle f_\lambda,\psi\rangle^2 d\lambda = \sum_{\lambda\in \mbox{spec}(M)}\left(\int_\Omega f_\lambda\psi  e^{-2\rho}dv_g \right)^2= \sum_{\lambda\in \mbox{spec}(M)}\left(\int_\Omega y_\lambda\psi  e^{-\rho}dv_g \right)^2$$
$$=\sum_{\lambda\in \mbox{spec}(M)}\left(\int_M y_\lambda\psi  e^{-\rho}dv_g \right)^2 =\int_M \psi^2  e^{-2\rho}dv_g = \Vert \psi\Vert^2,$$
where we used the same notation 
$\psi$ to designate the extension of $\psi$ by zero outside $\Omega$.

Let $R>0$ and let ${\mathfrak M}_0 =\{\lambda\in {\mathfrak M} \ :\ \lambda\le R\}$. 
Due to the transitive action of the isometry group $G$ on $M$, for every  eigenvalue $\Lambda$ of $\Delta_g$, with multiplicity $m_\Lambda$, the basis $\{y_\lambda \ : \lambda=\Lambda\}$ of the corresponding eigenspace is such that $\sum_{\lambda=\Lambda}y_\lambda ^2$ is constant on $M$. Integrating over $M$, we get 
\begin{equation}\label{sum-funct}
\sum_{\lambda=\Lambda}y_\lambda ^2 = \frac{m_\Lambda}{\vert M\vert_g}.
\end{equation}
Moreover, $0= \frac 12\Delta_g (\sum_{\lambda=\Lambda}y_\lambda ^2)= \sum_{\lambda=\Lambda}\left(\Lambda y_\lambda ^2 -\vert \nabla^gy_\lambda \vert^2\right)$, that is
\begin{equation}\label{sum-grad}
\sum_{\lambda=\Lambda}\vert \nabla^gy_\lambda \vert^2 = \frac{m_\Lambda}{\vert M\vert_g} \Lambda.
\end{equation}
Therefore, 
$$\int_{\mathfrak M_0}\Vert f_\lambda\Vert^2 d\lambda = \sum_{\lambda\le R}\int_\Omega y^2_\lambda  dv_g = \sum_{\Lambda\le R}  \frac{\vert \Omega\vert_g}{\vert M\vert_g} m_\Lambda =\frac{\vert \Omega\vert_g}{\vert M\vert_g} N(R)  .
$$
where $ N( R )$ is the number of eigenvalues of $\Delta_g$ on $M$ that are less or equal to $R$ (counted with multiplicity).
On the other hand, using \eqref{sum-funct} and \eqref{sum-grad}, we get for every  $\Lambda$,
$$\sum_{\lambda=\Lambda}\vert \nabla^gf_\lambda \vert^2 =e^{2\rho} \sum_{\lambda=\Lambda}\left( \vert \nabla^gy_\lambda \vert^2 + y_\lambda^2\vert \nabla^g \rho\vert^2  + g( \nabla^g \rho,  \nabla^g y_\lambda^2) \right)$$
$$= \frac{m_\Lambda}{\vert M\vert_g} e^{2\rho} \left(\Lambda +\vert \nabla^g \rho\vert^2  \right)  .$$
Thus
$$\int_{\mathfrak M_0}Q( f_\lambda, f_\lambda) d\lambda =\sum_{\lambda\le R}\int_\Omega \left( \vert \nabla^gf_\lambda \vert^2+ V  f_\lambda^2\right) we^{-2\rho} dv_g \qquad\qquad\qquad\qquad\qquad\qquad\qquad\qquad\qquad$$
$$\qquad\qquad\qquad= \sum_{\Lambda\le R}\int_\Omega  \frac{m_\Lambda}{\vert M\vert_g}  \left(\Lambda +\vert \nabla^g \rho\vert^2   + V  \right) w dv_g \frac{\int_\Omega w\, dv_g}{\vert M\vert_g}  \sum_{\lambda\le R} \lambda + \frac{\int_\Omega \widetilde{V}w \, dv_g}{\vert M\vert_g} N( R )  .
$$
Inserting into \eqref{RieszVersion}, we get for every $z\in \R$ and $R>0$, 
\begin{equation}\label{sum1}
\begin{split}
\sum_{j\ge0}(z-\mu_j)_+ & \ge z  \frac{\vert \Omega\vert_g}{\vert M\vert_g} N(R) -  \frac{\int_\Omega w\, dv_g}{\vert M\vert_g}  \sum_{\lambda\le R} \lambda - \frac{\int_\Omega \widetilde{V}w \, dv_g}{\vert M\vert_g} N( R )\\
&=\frac{\vert \Omega\vert_g}{\vert M\vert_g}\sum_{\lambda\le R}  \left( z -   \lambda \fint_\Omega w\, dv_g  - \fint_\Omega \widetilde{V}w \, dv_g\right).
\end{split}
\end{equation}
Notice that the RHS is negative if $z\le \fint_\Omega \widetilde{V}w \, dv_g$. Now, when $z> \fint_\Omega \widetilde{V}w \, dv_g$, we can  choose  $R=\tilde z=\frac {z- \fint_\Omega \widetilde{V}w \, dv_g}{\fint_\Omega w\, dv_g  }$ so that the last sum is taken over all eigenvalues $\lambda $ for which the involved  terms are nonnegative, thus
$$\sum_{j\ge0}(z-\mu_j)_+ \ge\frac{\vert \Omega\vert_g}{\vert M\vert_g} \sum_{j\ge0}\left(z-\lambda_j \fint_\Omega w\, dv_g  - \fint_\Omega \widetilde{V}w \, dv_g\right)_+.$$

Regarding the case of equality, it follows from   Theorem \ref{AvePrin} that equality holds in \eqref{Lieb-Th_homog} if and only if $f_\lambda\in E(z)$ for  $ \lambda\le \tilde z$ and $f_\lambda$ is orthogonal to $E_0(z)$ for $ \lambda> \tilde z$. Equivalently, $e^\rho F(\tilde z)\subset E(z)$ and, since $\mbox{Span} \{f_\lambda\ :\ \lambda>\tilde z\} $ is the orthogonal complement of $e^\rho F(\tilde z)$, $E_0(z)\subset e^\rho F(\tilde z)$. 
\end{proof}

As we have seen in the previous sections, our technique allows obtaining  bounds on  eigenvalue sums.  In order to simplify the statement of these bounds, we intoduce the following notation : Given any sequence $(a)=(a_k)_{k\ge 0}$ of real numbers, we set for   $p\in [1,+\infty)$,
$${\mathfrak S}_{(a)} (p)= \sum_{j=0}^{\lfloor p\rfloor-1} a_j + (p-\lfloor p\rfloor) a_{\lfloor p\rfloor},$$
 so that
when $p$ is an integer, ${\mathfrak S}_{(a)} ( p )$ is nothing but the sum of the first $p$ terms $a_0,\cdots, a_{p-1}$ of the sequence $(a)$.

\begin{thm}\label{sum-homog}
Let $(M,g)$ be a compact homogeneous
Riemannian manifold.
Let $(\mu) =(\mu_l)_{l\ge0} $  be the sequence of  eigenvalues defined  by  \eqref{NeumannDef}
on an open set $\Omega \subset M$.
Then,  for every $p\in [1,+\infty)$,
\begin{equation}\label{sum_homog}
 {\mathfrak S}_{(\mu)} ( p ) \le \frac{\vert \Omega\vert_g}{\vert M\vert_g}\ {\mathfrak S}_{(\tilde \lambda)} \left(\frac{\vert M\vert_g}{\vert \Omega\vert_g}\  p \right),
\end{equation}
where $ (\tilde\lambda)= (\tilde\lambda_l)_{l\ge0} $ is the sequence defined by $  \tilde\lambda_l= \lambda_l\fint_\Omega w \, dv_g +\fint_\Omega \widetilde{V}w \, dv_g$. 
Moreover, equality holds in \eqref{sum_homog} for some $p=k\in \N^*$ if and only if 
\begin{equation}\label{eq_sum_homog}
E_0(\mu_k)\subset e^\rho F_0(\lambda_{\breve k}) \ \mbox{ and } \  e^\rho F(\lambda_{ \hat k-1})\subset E(\mu_k).
\end{equation}
with ${\breve k= \Big\lfloor \frac{\vert M\vert_g}{\vert \Omega\vert_g} k \Big\rfloor} $ and $ \hat k= \Big\lceil \frac{\vert M\vert_g}{\vert \Omega\vert_g} k \Big\rceil$, where $\lfloor\ \rfloor$ and $\lceil\ \rceil$ denote the floor and the ceiling  functions, respectively.

\end{thm}

Observe that we have $F_0(\lambda_{\breve k}) \subset F(\lambda_{ \hat k-1})$ with equality if and only if $\lambda_{ \hat k-1} < \lambda_{\breve k})$.

\begin{proof} [Proof of Theorem \ref{sum-homog}]
As  mentioned in Remark \ref {rem_leg},  Legendre's  transform enables us to obtain \eqref{sum_homog} from \eqref {Lieb-Th_homog}. Alternatively, we can prove \eqref{sum_homog} using the averaged principle, which has the advantage of allowing us to characterize the case of equality. Indeed,  taking $z=\mu_k$ in \eqref{sum1}  we immediately get, $\forall R>0$ 
\begin{equation}\label{sum_homog1}
\sum_{j=0}^{k-1}\mu_j \le \frac{\vert \Omega\vert_g}{\vert M\vert_g}\sum_{j=0}^{N( R )-1}\tilde\lambda_j +\left( k-  \frac{\vert \Omega\vert_g}{\vert M\vert_g} N( R ) \right)\mu_k. 
\end{equation}
Denote by $1=N_0<N_1<N_2<\dots<N_j<\dots$  the values taken by the function $N( R )$, $R\in\R$, that is $N_j=m_0+m_1+\dots+m_j$.  The sequence of eigenvalues of $\Delta_g$ on $M$ is then  numbered as follows :
$$0=\lambda_0< \lambda_1=\lambda_2=\cdots = \lambda_{N_1-1}<\lambda_{N_1}=\cdots=\lambda_{N_2-1}<\lambda_{N_2}=\cdots $$
$$=\lambda_{N_j-1}<\lambda_{N_j}=\cdots= \lambda_{N_{j+1}-1}<\lambda_{N_{j+1}}=\cdots$$
 Let $q\in\N$ such that 
$$N_q\le \frac{\vert M\vert_g}{\vert \Omega\vert_g} k<N_{q+1}.$$
We consider the  inequality \eqref{sum_homog1}  with first  $N( R )=N_q$ and, then, $N( R )=N_{q+1}$.   We multiply the first inequality   by $\alpha= (N_{q+1}- \frac {\vert M\vert_g }{\vert \Omega\vert _g} k)/m_{q+1}$ and add the second inequality  multiplied by $1-\alpha = ( \frac {\vert M\vert_g }{\vert \Omega\vert_g} k-N_q)/m_{q+1}$ to get 
$$\sum_{j=0}^{k-1}\mu_j \le \frac{\vert \Omega\vert_g}{\vert M\vert_g} \sum_{j=0}^{N_q-1}\tilde\lambda_j + (1-\alpha)\frac{\vert \Omega\vert_g}{\vert M\vert_g}m_{q+1} \tilde\lambda_{N_q}  + k\mu_k - \frac{\vert \Omega\vert_g}{\vert M\vert_g}(\alpha N_q+(1-\alpha) N_{q+1} )\mu_k$$
$$=\frac{\vert \Omega\vert_g}{\vert M\vert_g} \sum_{j=0}^{N_q-1}\tilde\lambda_j + 
\left(  k -\frac{\vert \Omega\vert_g}{\vert M\vert_g}N_q\right)\tilde\lambda_{N_q}=\frac{\vert \Omega\vert_g}{\vert M\vert_g} \left(\sum_{j=0}^{N_q-1}\tilde\lambda_j +  \left(\frac{\vert M\vert_g}{\vert \Omega\vert_g}\  k -N_q\right)\tilde\lambda_{N_q}\right)$$
since $\alpha$ is chosen such that $\alpha N_q+(1-\alpha) N_{q+1}=\frac {\vert M\vert_g }{\vert \Omega\vert _g} k$. Now, from the definition of $N_q$ we have $\lambda_{N_q}=\lambda_{N_q+1}=\cdots= \lambda_{\lfloor \frac {\vert M\vert_g }{\vert \Omega\vert _g} k\rfloor}$ and, then,
$$\sum_{j=0}^{N_q-1}\tilde\lambda_j +  \left(\frac{\vert M\vert_g}{\vert \Omega\vert_g}\  k -N_q\right)\tilde\lambda_{N_q} = {\mathfrak S}_{(\tilde \lambda)} \left(\frac{\vert M\vert_g}{\vert \Omega\vert_g}\  k \right),$$
which yields   
$$\sum_{j=0}^{k-1}\mu_j \le \frac{\vert \Omega\vert_g}{\vert M\vert_g} {\mathfrak S}_{(\tilde \lambda)} \left(\frac{\vert M\vert_g}{\vert \Omega\vert_g}\  k \right).$$
This means that \eqref{sum_homog}  holds for all $p\in\N$. Since the functions ${\mathfrak S}_{(\mu)}$ and $ {\mathfrak S}_{(\tilde \lambda)}$ are piecewise-affine in $p$, the extension of \eqref{sum_homog} to all positive $p$ is immediate.

\smallskip
Let $k$ be a positive integer. The equality is achieved in \eqref{sum_homog} for $p=k$ if and only if one of the following holds :
\begin{itemize}
\item $ \frac{\vert M\vert_g}{\vert \Omega\vert_g} k=N_q$ and equality holds in \eqref{sum_homog1} for $R$ such that $N( R )=N_q$, i.e. for $R= \lambda_{N_q-1}$
\item $ \frac{\vert M\vert_g}{\vert \Omega\vert_g} k>N_q$ and equality holds in \eqref{sum_homog1} for the values of $R$ such that  $N( R )=N_q$ and $N( R )=N_{q+1}$, i.e. for both $R= \lambda_{N_q-1}$ and $R= \lambda_{N_{q+1}-1}$.
\end{itemize}
The first case corresponds to the case of equality in Theorem \ref{AvePrin} with $z=\mu_k$, ${\mathfrak M} =\mbox{spec}(M)$, ${\mathfrak M}_0 =\{\lambda\in \mbox{spec}(M)\ ; \ \lambda\le \lambda_{N_q-1}\}$. As in the proof of Theorem \ref{gen-homog}, this situation occurs if and only if $E_0(\mu_k)\subset e^\rho F(\lambda_{N_q -1})\subset E(\mu_k)$, with $N_q= \frac{\vert M\vert_g}{\vert \Omega\vert_g} k$. Since $\lambda_{N_q -1} <\lambda_{N_q }$,  $F(\lambda_{N_q -1}) = F_0(\lambda_{N_q })$, and the last conditions can  be written as follows :
\begin{equation}\label{eq_first_case}E_0(\mu_k)\subset e^\rho F_0(\lambda_{\frac{\vert M\vert_g}{\vert \Omega\vert_g} k}) \mbox{ and }\  F(\lambda_{\frac{\vert M\vert_g}{\vert \Omega\vert_g} k-1})\subset E(\mu_k).
\end{equation}
which is equivalent to \eqref{eq_sum_homog}.

\smallskip
In the second case, similar considerations show that equality holds if and only if 
\begin{equation}\label{eq_second_case}
E_0(\mu_k)\subset e^\rho F(\lambda_{N_{q} -1})\subset e^\rho F(\lambda_{N_{q+1} -1})\subset E(\mu_k).
\end{equation}
Since $N_q< \frac{\vert M\vert_g}{\vert \Omega\vert_g} k<N_{q+1}$, it is clear that 
$$N_q\le \Big\lfloor \frac{\vert M\vert_g}{\vert \Omega\vert_g} k \Big\rfloor \le N_{q+1}-1$$
and 
$$N_q\le  \Big\lceil \frac{\vert M\vert_g}{\vert \Omega\vert_g} k \Big\rceil -1 \le N_{q+1}-1.$$
Thus,
$$\lambda_{N_{q} }=\lambda_{N_{q+1} -1}= \lambda_{ \Big\lfloor \frac{\vert M\vert_g}{\vert \Omega\vert_g} k \Big\rfloor} = \lambda_{ \Big\lceil \frac{\vert M\vert_g}{\vert \Omega\vert_g} k \Big\rceil -1} .$$
Consequently, 
$$F(\lambda_{N_{q+1} -1})= F\Big(\lambda_{ \Big\lceil \frac{\vert M\vert_g}{\vert \Omega\vert_g} k \Big\rceil -1}\Big)$$
and, since $\lambda_{N_q -1}<\lambda_{N_q }$,
$$F(\lambda_{N_q -1}) = F_0(\lambda_{N_q })= F_0\Big({\lambda_{ \Big\lfloor \frac{\vert M\vert_g}{\vert \Omega\vert_g} k \Big\rfloor} }\Big).$$
Therefore, \eqref{eq_second_case} is equivalent to  \eqref{eq_sum_homog}.
\end{proof}

\begin{rems}  1. The particular case of \eqref{sum_homog} in which $w=1$ and $\rho=V=0$ corresponds to the inequality obtained by Strichartz  \cite[Theorem 2.2]{Str}.

In the same paper \cite{Str},  Strichartz also proved, following  Gallot \cite[Proposition 2.9]{Gal}, that for Dirichlet eigenvalues $\mu_l^D$ of the Laplacian on a domain $\Omega$ of a compact homogeneous
Riemannian manifold $(M,g)$, the reverse inequality 
$${\mathfrak S}_{(\mu^D)} ( p ) \ge \frac{\vert \Omega\vert_g}{\vert M\vert_g}\ {\mathfrak S}_{(\lambda)} \left(\frac{\vert M\vert_g}{\vert \Omega\vert_g}  p \right)$$
holds.  Contrary to what was found for Neumann eigenvalues in Theorem \ref{sum-homog}, a straightforward extension of the latter inequality to Dirichlet eigenvalues of  a Laplacian with potential cannot hold in general. Indeed, such an extension would imply for $p=1$ that $\mu_0^D(\Delta_g+V)\ge \frac 1{\vert M\vert_g} \int_\Omega {V} \, dv_g$, which is not always true  (for example, if $\Omega$ is a spherical cap of radius $r$ and if $u$ is a positive first  eigenfunction of the Dirichlet Laplacian on $\Omega$, we can take the family of continuous potentials $V_\varepsilon$ with $ V_\varepsilon=\frac1{u^2}$ on  the spherical cap of radius $(1-\varepsilon)r$, and $V_\varepsilon$ is constant on the complement, then, using $u$ as a test function, it is easy to see that  $\mu_0^D(\Delta_g+V_\varepsilon)\le \mu_0^D(\Delta_g) + \vert\Omega\vert $ while  $\int_\Omega {V_\varepsilon} \, dv_g$ tends to infinity as $\varepsilon \to 0$.)

\smallskip
\noindent 2. Assume that $\vert \Omega\vert \ge \frac 2{3}\vert M\vert$, then an immediate consequence of Theorem \ref{sum-homog} and the fact that the first  positive eigenvalue $\lambda_1$  of the Laplacian on a homogeneous manifold $(M,g)$ has multiplicity at least 2, is the following inequality 
$$\mu_0+\mu_1\le \frac {\vert \Omega\vert_g }{\vert M\vert_g } \left(2 \lambda_1\fint_\Omega w \, dv_g +3\fint_\Omega \widetilde{V}w \, dv_g \right)$$
which yields for the Neumann Laplacian (with $\rho=V=0$ and $w=1$)
 $$\mu_1\le 2\frac {\vert \Omega\vert_g }{\vert M\vert_g }  \lambda_1.$$

\end{rems}

In the case where $\Omega$ is equal to the whole of $M$, Theorem \ref{sum-homog} leads to the following 

\begin{cor}\label{closed_homog} Let $(M,g)$ be a compact homogeneous
Riemannian manifold.
Let $\mu_l $, $l\in\N$, be the  eigenvalues defined  by   \eqref{NeumannDef}
on $M$. 
Then, for every $k\in\N^*$,
\begin{equation}\label{sum_closed_homog}
 \sum_{j=0}^{k-1} \mu_j \le \sum_{j=0}^{k-1} \tilde\lambda_j,
\end{equation}
where equality holds if and only if 
$$E_0(\mu_k)\subset e^\rho F_0(\lambda_k)\  \mbox{ and  } \ e^\rho F(\lambda_{k-1})\subset E(\mu_k).$$
In particular, 
if $m_1$ is the multiplicity of $\lambda_1$, then equality holds in \eqref{sum_closed_homog} for $k\le m_1$ if and only if 
$(V+\vert\nabla^g\rho\vert^2 )w -\mbox{div}_g(w\nabla\rho)$ is constant on $M$ and $\mu_j=\tilde\lambda_j$ for $j=0,1,\cdots, k-1$.

 \end{cor}
 
 \begin{proof}[Proof of Corollary \ref{closed_homog}] 
 Assume that equality holds in \eqref{sum_closed_homog} for $k\le m_1$. Then we have  $E_0(\mu_k)\subset e^\rho F_0(\lambda_k)$. Since $k\le m_1$, $F_0(\lambda_k)= F(\lambda_0)=\mbox{span} \{1\}$. It follows that $E_0(\mu_k)$ has dimension 1, that is $E_0(\mu_k)=E(\mu_0) =\mbox{span} \{e^\rho\}$. Consequently, $\mu_1=\mu_2=\cdots=\mu_k$, and $e^\rho$ is an eigenfunction of $H$ associated with $\mu_0$. Thus
 $$He^\rho=  e^{2 \rho}\mbox {div}_g\left(w e^{-2 \rho} \nabla^g e^\rho\right)   + V we^\rho=\cdots,$$
 $$\qquad \qquad=\left((V+\vert\nabla^g\rho\vert^2 )w -\mbox{div}_g(w\nabla\rho) \right) e^\rho=\mu_0 e^\rho$$
 which implies that $(V+\vert\nabla^g\rho\vert^2 )w -\mbox{div}_g(w\nabla\rho)=\mu _0$. Integrating, we get $\mu_0 = \fint_\Omega \widetilde{V} w \, dv_g=\tilde\lambda_0$. 
 Now, 
 $$\mu_0+ (k-1)\mu_1= \sum_{j=0}^{k-1} \mu_j = \sum_{j=0}^{k-1} \tilde\lambda_j=\tilde \lambda_0  +(k-1)\tilde \lambda_1$$
 and, consequently, $\mu_1=\tilde\lambda_1$.

\end{proof} 

\begin{rems}
1. An immediate consequence of Corollary \ref{closed_homog} is that for any potential $V$ on   a compact homogeneous $(M,g)$, one has for every positive $k$,
$$\frac 1k \sum_{j=0}^{k-1} \mu_j (\Delta_g +V)\le \frac 1k\sum_{j=0}^{k-1} \mu_j (\Delta_g)+\fint_M V \, dv_g,$$
to be compared with the results of \cite{EI1}.

\smallskip
\noindent 2. We know that for any $k\ge 2$, either $\lambda_{k-1}= \lambda_k$ or else $\lambda_{k-1}= \lambda_{k-2}$.
Notice that if  $\lambda_{k-1}= \lambda_k$ and if equality holds in \eqref{sum_closed_homog}  for $k$, then, necessarily, $\tilde \lambda_{k-1}=\mu_{k-1}= \mu_k=\tilde \lambda_k$.
(This follows directly from the combination of $\sum_{j=0}^{k-1} \mu_j = \sum_{j=0}^{k-1} \tilde\lambda_j$ with $ \sum_{j=0}^{k} \mu_j \le \sum_{j=0}^{k} \tilde\lambda_j $ and $ \sum_{j=0}^{k-2} \mu_j \le \sum_{j=0}^{k-2} \tilde\lambda_j$.)  Consequently, the equality also holds in \eqref{sum_closed_homog}  for $k-1$ and $k+1$.

Moreover,  if $\mu_k>\mu_{k-1}$, then the equality  holds in \eqref{sum_closed_homog}  for $k$ if and only if  $\lambda_k>\lambda_{k-1}$ and $ E(\mu_{k-1}) = e^\rho F(\lambda_{k-1})$.  (Indeed, in this case, $\dim E_0(\mu_k)= \dim E(\mu_{k-1})= k$ and $\dim F_0(\lambda_k)= \dim F_0(\lambda_{k-1})=k$.)

 \end{rems}

Applying the Laplace transform to both sides of \eqref {Lieb-Th_homog}, we obtain the following comparison of the heat traces (see \eqref{laplace}):

\begin{cor}\label{heat-homog}
Let $(M,g)$ be a compact homogeneous
Riemannian manifold.
Let $\mu_l $, $l\in\N$, be the    eigenvalues defined by  \eqref{NeumannDef}
on a bounded open set $\Omega \subset M$.
Then,  for all $t>0$,
\begin{equation}\label{heat_homog}
 \sum_{j\ge 0} e^{-\mu_j t}\ge \frac{\vert \Omega\vert_g}{\vert M\vert_g}  \sum_{j\ge 0} e^{-\tilde\lambda_j t},
\end{equation}
where $ \tilde\lambda_j  = \lambda_j \fint_\Omega w \, dv_g +\fint_\Omega \widetilde{V}w \, dv_g  $.

\end{cor}

Let us define the theta function via:
$$\Theta ( t ) = \frac 1{4\pi t}\sum _{(p,q)\in\mathbb{ Z}^2} e^{-\frac{p^2+q^2+pq}{4t}}.$$


\begin{cor}\label{heat-torus}
Let $\Gamma=\Z e_1\oplus\Z e_2\subset\R^2$ be a lattice, where $\{e_1,e_2\}$ is a basis of $\R^2$. Let $\rho$, $w>0$, $V$ be $\Gamma$-periodic functions on $\R^2$ and denote by $\mu_l=\mu_l( \rho,w,V)$, $l\in\N$, the eigenvalues of the operator $H( \rho,w,V)$ defined by \eqref{Operatror1}, acting on $\Gamma$-periodic functions on  $\R^2$. Then, for all $t>0$,
\begin{equation}\label{heat_periodic}
 \sum_{j\ge 0} e^{-\mu_j t}\ge   \Theta \left(\frac{\fint_\Omega w \, dx}{\vert \Omega\vert} t \right)   e^{-t\fint_\Omega \widetilde{V}w \, dx},
\end{equation}
where $\Omega$ is a fundamental domain for the action of $\Gamma$ on $\R^2$.

\end{cor}
\begin{proof} 
This result is a direct consequence of \eqref{heat_homog} combined with Poisson's formula and  Montgomery's Theorem \cite{Mont}. 

\end{proof} 

We turn now to the phase space analysis taking into account the form of the potential $V$, and allowing conformal transformations and nontrivial weights.
{
Let us denote by
$$
0 = \Lambda_0< \Lambda_1 < \Lambda_2 < \cdots < \Lambda_l < \cdots
$$
the increasing sequence of eigenvalues of the Laplacian of the compact homogeneous space $(M,g)$. The multiplicity of $\Lambda_l$ is denoted $m_l$, and we designate by $\{y_{l,1}, y_{l,2},\cdots,y_{l,m_l}\}$  an $L^2$-orthonormal basis of the eigenspace associated with $\Lambda_l$.
}

In the case of a domain
$\Omega$ in a manifold $X \simeq (M, e^{-2 \rho} g)$ that is conformally equivalent to  $(M,g)$, we shall use
coherent-state test functions  of the form:
\begin{equation}\label{cptcohdef}
f_\zeta({\bf x}) := y_{\ell \, m}({\bf x}) e^{\rho({\bf x})} h_{\bf y}({\bf x}).
\end{equation}
In this formula, $h({\bf x})$ is
a nonnegative $H^1$ function supported in the
geodesic ball of radius $r$ in the canonical metric on $M$,
with $\int_{B_r}{h_r^2({\bf x}) d^\nu\,x} = 1$, and $\bf y$ ranges over the
isometry group $G$.  
As before we choose it specifically as the ground-state Dirichlet eigenfunction 
on the geodesic ball of radius $r$ and set
$ \quad\mathcal{K}(h_r) := \int_{B_r}{|\nabla h({\bf x})|^2 d^\nu\,x}$,
which is thus the fundamental Dirichlet eigenvalue
for the Laplacian on the geodesic disk of radius $r$.
Denoting by
$T_{\bf y}({\bf x})$ the 
{action by the group element ${\bf y}$ on}
the point ${\bf x}$, we let
$$
h_{\bf y}({\bf x}) := h(T_{\bf y}({\bf x})).
$$
Recall that one can designate an arbitrary point of 
$M$ as $0$ and cover $M$ with translates $T_{\bf y}(0)$.
We normalize the uniform measure $d \gamma$ on $G$
so that for any
$f \in L^\infty(M)$, $\int_{G}{f(T_{\bf y}({\bf x})) d \gamma({\bf y})} =
\int_{M}{f({\bf x}) d v_g}$.  The index
$\zeta = (\ell,m,{\bf y})$ ranges over
$\mathfrak{M} = \mathcal{J} \times G$,
where $\mathcal{J}$ is the set of all pairs of integer indices
for the normalized eigenfunctions
$y_{\ell\,m}({\bf x})$, and the associated measure $d \sigma$ is the product of the counting measure on
$\mathcal{J}$ with $d \gamma$.

As in Section 2, we find it helpful to define:

\begin{defi}
As before,
$$
\widetilde{V}({\bf x}) := V({\bf x}) + |\nabla \rho|^2,
$$

\noindent
The
{\em weighted phase-space volume} is
\begin{align}
\Phi_w^h(\Lambda) &:= \Big|\{\ell, m, {\bf y}\}: m \le m_\ell, T_{\bf y}(0) \in \Omega,\Lambda_\ell + \widetilde{V}(T_{\bf y}(0)) \le \Lambda \Big|\nonumber\\
&= \int_{\{{\bf y}: T_{\bf y}(0) \in \Omega, \widetilde{V}(T_{\bf y}(0)) \le \Lambda\}}{\left(\sum_{\{\ell: \Lambda_\ell + \widetilde{V}(T_{\bf y}(0)) \le \Lambda\}}{m_\ell
}\right)d\gamma({\bf y})}.
\end{align}

\noindent
The {\em total energy} associated with this phase-space volume is correspondingly
\begin{equation}
E_w^h(\Lambda) := \int_{\{{\bf y}: T_{\bf y}(0) \in \Omega,\widetilde{V}(T_{\bf y}(0)) \le \Lambda\}}{\left(\sum_{\{\ell: \Lambda_\ell + \widetilde{V}(T_{\bf y}(0)) \le \Lambda\}}{m_\ell
\left(\Lambda_\ell + \widetilde{V}(T_{\bf y}(0))\right)}\right)d\gamma({\bf y})}.
\end{equation}
\end{defi}

\begin{thm}\label{homogphasespace}
Let $\mu_0 \le \mu_1 \le \dots$ be the variationally defined Neumann eigenvalues \eqref{NeumannDef}
on a bounded open set $\Omega \subset M$, where
$w$, $\rho$, and $V$ satisfy the assumptions stated in Section 1.
Then for all $r > 0$,
\begin{equation}\label{HomogUpperbd}
\sum_{j=0}^{k-1} \mu_j \le
E_w^h(\Lambda + {\rm Lip}(\Lambda) r) + \mathcal{K}(h_r) \Phi_w^h(\Lambda + {\rm Lip}(\Lambda) r).
\end{equation}
\end{thm}

\begin{proof}
Note that
$$
\left\langle{\phi, f_\zeta}\right\rangle_{\Omega} = \left\langle{e^{-\rho({\bf x})} h_{\bf y}({\bf x}) \phi({\bf x}), y_{\ell \, m} }\right\rangle_{M},
$$
where if $\Omega$ is a strict subset of $M$, then $\phi$ is extended by $0$
outside $\Omega$.
Thus, by the Fourier completeness relation,
\begin{align}\label{Parseval-Y}
\int_{G}\left(\sum_{\ell, m}{\left|\left\langle\phi, f_\zeta\right\rangle_\Omega \right|^2 }\right) d \gamma({\bf y})
&=
\int_{G}{\| e^{-\rho({\bf x})} h_{\bf y}({\bf x}) \phi({\bf x})\|^2_{L^2(M, dv_g)}}d \gamma({\bf y}) \nonumber\\
&=
\int_{\Omega}{\left|\phi\right|^2 e^{-2 \rho} \left(\int_{G} h_{\bf y}({\bf x})^2 d \gamma({\bf y})\right) dv_g}\nonumber\\
&=
\int_{\Omega}{\left|\phi\right|^2 e^{-2 \rho}  dv_g}= \|\phi\|^2.
\end{align}
To apply the theorem, choose $\mathfrak{M}_0$ of the form
$\{(\ell,m, {\bf y}):m \le m_ \ell , T_{\bf y}(0) \in \Omega, \Lambda_ \ell  + \widetilde{V}(T_{\bf y}(0)) \le \Lambda\}$ 
for a finite $\Lambda$ large enough so that
\begin{align}\label{EucCond2}
k \le \int_{\mathfrak{M}_0}{\|f_\zeta\|^2_{L^2(\Omega)} d\sigma}
= \int_{\Omega}\int_{\mathfrak{M}_0}{h^2_{\bf y}({\bf x}) e^{2 \rho({\bf x}) - 2 \rho({\bf x})}d\sigma dv_g} 
= |\Omega| \Phi_w^h(\Lambda).
\end{align}
We define $\Lambda(k)$ as the minimal value of  $\Lambda$ for which \eqref{EucCond2} is valid and henceforth 
choose $\mathfrak{M}_0 = \{(\ell,m, {\bf y}):m \le m_ \ell , T_{\bf y}(0) \in \Omega, \Lambda_ \ell  + \widetilde{V}({\bf y}) \le \Lambda(k)\}$.
Then
$$
\sum_{j=0}^{k-1} \mu_j \le \int_{\mathfrak{M}_0}{\mathcal{R}(f_\zeta) d\sigma(\zeta)} \quad\quad\quad\quad\quad\quad\quad\quad\quad\quad
\quad\quad\quad\quad\quad\quad\quad\quad\quad\quad
$$
\begin{align}
\quad\quad\quad\quad&=\int_{\Omega}\int_{\mathfrak{M}_0}
{w({\bf x})
\left(y_{\ell m}^2\left(h_{\bf y}^2({\bf x}) \widetilde{V}({\bf x}) + |\nabla h_{\bf y}({\bf x})|^2 +
\nabla \rho({\bf x}) \cdot \nabla h^2_{\bf y}({\bf x}) \right)\right.
}\nonumber\\
&\quad\quad\quad\quad\quad{+ \left. h^2_{\bf y}({\bf x}) |\nabla y_{\ell m}|^2 + 2 h_{\bf y}({\bf x}) \nabla h_{\bf y}({\bf x}) \cdot y_{\ell m} \nabla y_{\ell m}\right) d \sigma dv_g} \nonumber
\end{align}
\begin{equation}
\quad\quad\quad\le\int_{\Omega}\int_{\{{\bf y}: \widetilde{V}(T_{\bf y}(0)) \le \Lambda\}}
{w({\bf x})\left(\sum_{\{\ell: \Lambda_\ell + \widetilde{V}(T_{\bf y}(0)) \le \Lambda\}}
\left(\frac{m_\ell}{|M|}\right)\left(h_{\bf y}^2({\bf x}) (\Lambda_\ell + \widetilde{V}({\bf x})) + |\nabla h_{\bf y}({\bf x})|^2\right)\right)d\gamma({\bf y})},\nonumber
\end{equation}
\normalsize
by dint of \eqref{sum-funct} and \eqref{sum-grad}.  (The
final cross term dropped out because it was proportional to
the gradient of a constant function \eqref{sum-funct}, in analogy with \eqref{cancel}.)  
Because $h$ is supported in a ball of radius $r$, we restrict the $x$-integration to
${\bf x} : dist({\bf x}, {\bf y}) \le r$ with $(\ell, m, {\bf y}) \in \mathfrak{M}_0$ and estimate the integral in analogy with \eqref{Eucl bd},
obtaining

\begin{equation}\label{cptupper}
\sum_{j=0}^{k-1} \mu_j \le \left(\frac{1}{|M|}\right)\int_{(\ell,m,{\bf x}) \in \mathfrak{M}_0(\Lambda + {\rm Lip}(\Lambda) r)}
{w({\bf x})
\left(\Lambda_\ell + \widetilde{V}({\bf x}) + \mathcal{K}(h_r)\right)d\sigma},
\end{equation}
which yields the statement in the Theorem.
\end{proof}


{\bf Acknowledgments}
E.H. is grateful to the Universit\'e F. Rabelais and to
\'Ecole Polytechnique F\'ed\'erale de Lausanne for
hospitality that supported this collaboration.


\def\cprime{$'$} \def\cprime{$'$}

\end{document}